\documentclass[sn-mathphys,Numbered]{sn-jnl}

\usepackage{graphicx}%
\usepackage{multirow}%
\usepackage{amsmath,amssymb,amsfonts}%
\usepackage{amsthm}%
\usepackage{mathrsfs}%
\usepackage[title]{appendix}%
\usepackage{xcolor}%
\usepackage{textcomp}%
\usepackage{manyfoot}%
\usepackage{booktabs}%
\usepackage{algorithm}%
\usepackage{algorithmicx}%
\usepackage{algpseudocode}%
\usepackage{listings}%

\usepackage{subfigure}

\theoremstyle{thmstyleone}%
\newtheorem{theorem}{Theorem}
\newtheorem{proposition}[theorem]{Proposition}%

\theoremstyle{thmstyletwo}%
\newtheorem{example}{Example}%
\newtheorem{assumption}{Assumption}
\newtheorem{lemma}{Lemma}

\theoremstyle{thmstylethree}%
\newtheorem{definition}{Definition}%

\newcommand{\A}{{\mathcal A}}
\newcommand{\B}{{\mathcal B}}

\newcommand{\M}{{\mathcal M}}
\newcommand{\G}{{\mathcal G}}

\newcommand{\z}{{\mathbf z}}

\newcommand{\x}{{\mathbf x}}
\newcommand{\y}{{\mathbf y}}
\newcommand{\e}{{\mathbf e}}
\newcommand{\uu}{{\mathbf u}}
\newcommand{\vv}{{\mathbf v}}

\newcommand{\0}{{\mathbf 0}}
\newcommand{\I}{{\mathbb I}}
\newcommand{\LL}{{\mathcal L}}
\newcommand{\Pp}{{\mathcal P}}

\def\SS{\operatorname{S}}
\def\Rn{\mathbb R^n}

\def\rk{\operatorname{rank}}
\def\blue#1{\textcolor{blue}{#1}}

\begin{document}

\title[Third order symmetric low rank approximation]{Quantifying low rank approximations of third order symmetric tensors}


\author[1]{\fnm{Shenglong} \sur{Hu}}\email{shenglonghu@hdu.edu.cn}

\author*[2]{\fnm{Defeng} \sur{Sun}}\email{defeng.sun@polyu.edu.hk}

\author[3]{\fnm{Kim-Chuan} \sur{Toh}}\email{mattohkc@nus.edu.sg}

\affil[1]{\orgdiv{Department of Mathematics}, \orgname{Hangzhou Dianzi University}, \orgaddress{\city{Hangzhou}, \postcode{310018}, \state{Zhejiang}, \country{China}}}

\affil[2]{\orgdiv{Department of Applied Mathematics}, \orgname{The Hong Kong Polytechnic University}, \orgaddress{\city{Hung Hom},  \state{Hong Kong}}}

\affil[3]{\orgdiv{Department of Mathematics}, \orgname{National University of Singapore}, \orgaddress{\street{10 Lower Kent Ridge Road}, \country{Singapore}}}


\abstract{In this paper, we present a method to certify the approximation quality of a low rank tensor to a given third order symmetric tensor. Under mild assumptions, best low rank approximation is attained if a control parameter is zero or quantified quasi-optimal low rank approximation is obtained if the control parameter is positive.
This is based on a primal-dual method for computing a low rank approximation for a given tensor. The certification is derived from the global optimality of the primal and dual problems, and is characterized by easily checkable relations between the primal and the dual solutions together with another rank condition. The theory is verified theoretically for orthogonally decomposable tensors as well as numerically through examples in the general case.}




\maketitle

\section{Introduction}\label{sec:introduction}
Let $\mathbb R$ be the field of real numbers and $\SS^m(\mathbb R^n)$ be the space of symmetric tensors with real entries of order $m$ and dimension $n$ for positive integers $m$ and $n$.
When $m=2$, $\SS^2(\mathbb R^n)$ is the space of symmetric $n\times n$ real matrices. A tensor $\A\in \SS^m(\mathbb R^n)$ can be represented by its entries $a_{i_1\dots i_m}$ with $i_j\in\{1,\dots,n\}$ for all $j\in\{1,\dots,m\}$. In this paper, we will mainly focus on third order symmetric tensors, i.e., $m=3$, which is the most important case in several applications \cite{lim2021tensors}. In the set $\operatorname{S}^3(\mathbb{R}^n)$, there is the fundamental set of \textit{decomposable tensors}, i.e., tensors of the form $\x^{\otimes 3}$ for a vector $\x\in\mathbb R^n$, where $\x^{\otimes 3}$ is a short hand for
\[
\x\otimes\x\otimes \x\in \SS^3(\mathbb R^n)
\]
whose $(i_1,i_2,i_3)$-entry is $x_{i_1}x_{i_2}x_{i_3}$ for all $i_1,i_2,i_3\in\{1,\dots,n\}$.

A third order symmetric tensor $\mathcal A$ has (real symmetric) \textit{rank} $r$ if it can be represented as (cf.\ \cite{CGLM-08})
\begin{equation}\label{eq:non-decomp}
\mathcal A=\sum_{i=1}^r\lambda_i\mathbf x_i^{\otimes 3}\ \text{for some }\lambda_i>0\ \text{and }\|\mathbf x_i\|=1\ \text{with }i=1,\dots,r,
\end{equation}
and a decomposition of the form \eqref{eq:non-decomp} with a summand strictly smaller than $r$ does not exist. The rank of a tensor $\A$ is denoted as $\rk(\A)$.

A rank one decomposition of a given $\A\in\SS^3(\Rn)$ is a decomposition as in \eqref{eq:non-decomp}, it becomes a \textit{rank decomposition} if $r=\rk(\A)$.
It is a fact that each tensor $\A\in\SS^3(\Rn)$ has a rank decomposition as \eqref{eq:non-decomp} \cite{CGLM-08}.

In this paper, we consider the following problem,
which is termed as the \textit{best rank-$r$ approximation problem}:
\begin{equation}\label{eq:best-r}
\inf\bigg\{\frac{1}{2}\big\|\mathcal A-\B\big\|^2 : \rk(\B)\leq r\bigg\},
\end{equation}
for a given tensor $\A\in\SS^3(\mathbb R^n)$ and a positive integer $r$. Usually, $r$ is assumed to be small, and thus it is also called the \textit{best low rank approximation problem}.

The prefix ``inf" in the optimization problem \eqref{eq:best-r} instead of ``min" is due to the fact that the constraint set is not necessarily closed and hence the optimizer may not exist, which is emphasized by De Silva and Lim \cite{DL-08}.
Therefore, the following assumption is necessary.
\begin{assumption}\label{assmp:well-defined}
There exists a best rank-$r$ approximation for the given tensor $\A$, i.e., problem~\eqref{eq:best-r} has an optimal solution.
\end{assumption}

Assumption~\ref{assmp:well-defined} is assumed throughout this paper. Under Assumption~\ref{assmp:well-defined}, the ``$\inf$" in problem \eqref{eq:best-r} can be strengthened as ``$\min$".

Tensors (a.k.a. hypermatrices) have become a standard tool in a wide variety of applications and mathematical sciences \cite{Lim:hyp,L-12,Kol-Bad:sur,lim2021tensors}. Particularly, it is one of the foundations in multilinear algebra. A fundamental topic in this area is tensor decomposition, which finds a decomposition of the given tensor via rank one tensors with the smallest possible length (cf.\ \eqref{eq:non-decomp}). A companion but with the same importance is the best low rank approximation for a given tensor, which is a cornerstone in several disciplines, especially in applications where data is collected with noise and thus approximation is prevalent \cite{GKT-13}.
The best low rank approximation problem for a given tensor has a rich literature, see \cite{Kol-Bad:sur,CGLM-08,De-De-Van:bes,GKT-13,L-12,C-00,hu2023linear} and references therein.

In this paper, we will focus on symmetric tensors. A second order symmetric tensor is the symmetric matrix in the usual sense. However, many crucial properties change dramatically if we switch from second order symmetric tensors to higher orders, e.g., the possible nonexistence of optimizers for \eqref{eq:best-r} \cite{DL-08}. Behind this unfavorable phenomena is that the rank of a tensor has more complicated but rich features, such as the NP-hardness to compute it as established by H\aa stad \cite{Has:ran}, the nonsymmetric rank differs from the rank as noted by Shitov \cite{S-18}; we refer the reader to \cite{CGLM-08,L-12} and references therein for more information. However, despite of these mysteries, many favourable properties, such as the celebrated Alexander-Hirschowitz's theorem \cite{AH-95} which identifies all the generic ranks, and Kruskal's theorem on the uniqueness of the rank decomposition \cite{K-77}, make methods based on symmetric tensor approximation/decomposition unbelievably appealing in a wide range of applications, such as blind sources separations \cite{CJ-10}.

The developments of the symmetric tensor approximation problem are based on the decomposition problem, cf.\ \eqref{eq:best-r}.
The symmetric tensor decomposition problem, also known as the \textit{Waring decomposition} in the algebraic geometry literature \cite{VW-02}, has attracted considerable attentions even very recently, such as Balllico and  Bernardi \cite{BB-12}, Bernardi, Gimigliano and Id\`a \cite{BGI-11}, Brachat, Comon, Mourrain and  Tsigaridas \cite{BCMT-10}, Comon and Mourrian \cite{CM-96}, Nie \cite{N-17}, and Oeding and Ottaviani \cite{OO-13}, etc. 

The other side of the coin is the symmetric tensor best low rank approximation problem. It is particularly preferable in applications, see \cite{Kol-Bad:sur,GKT-13}. For example, a low rank approximation can reduce the complexity of manipulating an $m$-th order symmetric tensor of dimension $n$ from $O(n^m)$ to $O(n)$, which is important and indispensable for applications where $n$ is large and computational cost such as $O(n^m)$ is prohibitive. The matrix case (i.e., $m=2$) is resolved by well-developed techniques in both theory and algorithms \cite{GV-12}. However, in contrast to the matrix counterpart, the higher order equivalents are still under investigation and several key problems must be resolved first. Besides the possible nonexistence of optimal solutions \cite{DL-08}, there are challenging issues such as the NP-hardness even just to obtain the best rank one approximation \cite{HL-13}, etc. Therefore, it is an important research topic in the community. For the case $r=1$, i.e., the best rank one approximation, there are extensive research done, such as Kofidis and Regalia \cite{KR-02}, Kolda and Mayo \cite{KM-11}, Nie and Wang \cite{NW-14}, Qi \cite{Qi:rat}, Zhang, Ling and Qi \cite{ZLQ-12}, etc. For general $r$, alternating minimization techniques are adopted to solve \eqref{eq:best-r}, see the survey \cite{C-00,GKT-13} and references therein. While very good numerical performances were observed, \textit{there is a lack of theoretical justification on what is found, such as global optimality certification, approximation quality guarantees, etc}, which are long standing fundamental questions for best low rank approximation. 
A very recent progress is made by Nie \cite{N-17-2}, in which a method is proposed using generating polynomials and it can find a quasi-optimal solution if the given tensor is sufficiently close to a low rank one.

In this paper, we will study the fundamental question in  symmetric tensor best low rank approximation:
\[
\textit{Certify a candidate as being globally optimal or as a quantified quasi-optimal solution.}
\]
For the rank one case (i.e., $r=1$), Nie and Wang \cite{NW-14} proposed a semidefinite relaxation method utilizing advances from polynomial optimization. It is shown that under a rank condition for the relaxed semidefinite programming (SDP) problem, it can certify the computed rank one tensor as being the best rank one approximation for the given data. The rank one approximation under nonnegativity is studied by the authors in \cite{hu2019best}. This paper considers the case $r\geq 1$. 

\textbf{Contributions.}
Our main theorem can be stated as follows, in which we defer the exact meaning of the notations in the content for clarity.
\begin{theorem}[Rank-$r$ Approximation]\label{thm:best-intro}
Let $\sigma\geq 0$ and $r\geq 1$. Given a tensor $\mathcal B=\sum_{i=1}^r\lambda_i\mathbf{x}_i^{\otimes 3}$. Let $\mathbf{y}$ be the moment sequence defined by the measure $\mu:=\sum_{i=1}^r\lambda_i\delta_{\mathbf{x}_i}$. If there exists a triplet $(U,W,Z)$ such that
\begin{enumerate}
\item the feasibility holds
\[
\mathcal M_k^*(Z)+\mathcal P_k^*(U)-\mathcal L_k^*(W)=\sigma\mathcal M_k^*(E_0),\ \text{and }Z\succeq 0,
\]
\item the optimality holds
\begin{equation*}
\langle Z,\mathcal M_k(\y)\rangle=0,\
\text{and }\\ \mathcal P_k(\y)\in\Pi_{\operatorname{R}(r)}(M(\A)-U).
\end{equation*}
\end{enumerate}
Then, it holds
\begin{enumerate} 
\item if $r=1$ and $\sigma<\rho(\A)$, then $(\lambda_1+\sigma)\x_1^{\otimes 3}$ is a best rank one approximation of $\A$,
\item if $\sigma<\frac{\tau(\A)\rho(\A)}{2r}$ and Assumption~\ref{assmp:well-condition} holds when $\sigma>0$, then $\mathcal B$ is a $2\sqrt{\frac{r}{\tau(\A)}}\bigg(\Big(1-\sqrt{\frac{\tau(\A)}{r}}\Big)\|\A\|+2\sigma\bigg)\sigma$-quasi-optimal rank-$r$ approximation of $\mathcal A$.
\end{enumerate}
\end{theorem}
Note that when $\sigma=0$, we get best rank-$r$ approximants.
We also note that $\tau(\A)$ is an intrinsic number determined by a rank decomposition of $\mathcal{B}$, which depends on $r$, and $\tau(\A)=1$ when $r=1$.

Theorem~\ref{thm:best-intro} is proved by first establishing a primal-dual method \eqref{eq:non-opt-rank-reform}-\eqref{eq:lagrange-dual-min}  for solving a relaxation of the best low rank approximation problem \eqref{eq:best-r}, and then an approximation quality analysis of the primal problem \eqref{eq:non-opt-rank-reform} to \eqref{eq:best-r}.

More precisely, we apply the semidefinite relaxation for moment problems and a rank characterization to reformulate the third order symmetric tensor best low rank approximation problem as a rank constrained nonlinear matrix optimization problem.
Advantages of nonsmooth analysis on low rank projection of matrices is explored to propose a relaxation for the reformulation, which is the \textit{primal} problem.  Then, the \textit{dual} problem is explicitly given. A theorem on the certification of the global optimality of a candidate solution for the primal and dual problem is given. One ingredient is that the global optimality certification conditions are explicitly given and easily checkable once a feasible solution for the primal problem, together with a Lagrange multiplier which is always available from a primal-dual algorithm, is given.
This is the result of our careful design of the primal problem \eqref{eq:non-opt-rank-reform}, while keeping in mind of a target dual certificate.
Approximation quality of problem \eqref{eq:non-opt-rank-reform} to the original best low rank approximation problem \eqref{eq:best-r} is then established. It depends on a control parameter $\sigma$. If $\sigma=0$, it relates to the best low rank approximation, and if $\sigma>0$,
a quantified quasi-optimal low rank approximation is given. Positive $\sigma$ is preferable
to ensure the strict feasibility of the
dual problem, see Lemma~\ref{lem:strict}. While, both $\sigma=0$ and $\sigma>0$ are allowable for approximation quality certifications, see Theorem~\ref{thm:best-intro}. The validation of this approach is verified for orthogonally decomposable tensors from a theoretical perspective, and several examples numerically.

Our method employs recent advances
from the semidefinite relaxation for the moment problems, duality theory in low rank matrix optimization and nonsmooth analysis for low rank matrix projection. In particular, it is directly motivated by the works by Gao and Sun on low rank matrix optimization problems and the duality theory \cite{GS-10,gao2010structured}, by Nie on semidefinite relaxation of moment problems \cite{N-14-ATKMP,nie2023moment,N-15}, and the tensor decomposition method by Tang and Shah \cite{TS-15}, in which a theoretical justification for the decomposition of orthogonally decomposable tensors is established.

\textbf{Contents.} The approximation problem and related preliminaries are given in Section~\ref{sec:tensor}. The nonlinear matrix optimization reformation is given in Section~\ref{sec:matrix}. In order to keep the main theme of this paper on the tensor best low rank approximation, some notations and supporting techniques for the nonsmooth analysis of low rank matrix projections and others are put in Appendix~\ref{sec:indices} and Appendix~\ref{app:low-rank}. Section~\ref{sec:app-quality} studies the approximation quality of the problem proposed in Section~\ref{sec:matrix}. The theoretical certification is given in Section~\ref{sec:app-quality} and the numerical illustration is given in Section~\ref{sec:numerical}. Some final remarks are given in Section~\ref{sec:conclusion}.

\section{Preliminaries}\label{sec:tensor}
In this paper, we will focus on third order symmetric tensors.
Given a positive integer $n$, a third order symmetric tensor $\mathcal A\in\operatorname{S}^3(\mathbb R^n)$ is a collection of $n^3$ scalars $a_{i_1i_2 i_3}$, termed the entries of $\mathcal A$, for all $i_j\in\{1,\dots,n\}$ and $j\in\{1,2,3\}$. As the case of symmetric matrices, the number of independent entries are smaller than, but in the same order of, $n^3$ due to the symmetry. There are all together ${n+2\choose n-1}$ independent entries encoded by a third order $n$ dimensional symmetric tensor.
A symmetric rank one tensor in $\operatorname{S}^3(\mathbb R^n)$ is an element $\mathbf x^{\otimes 3}$ for some vector $\mathbf x \in\mathbb R^n\setminus\{\0\}$.
We refer the readers to \cite{Lim:hyp} and references herein for basic notions on tensors.

The notation $\|\cdot\|$ represents the \textit{Euclidean norm} for a vector, the \textit{Frobenius norm} for a matrix \cite{HJ-85}, and the \textit{Hilbert-Schmidt norm} for a tensor \cite{Lim:hyp}, defined as
\[
\|\A\|:=\Big(\sum_{i,j,k}a_{ijk}^2\Big)^{\frac{1}{2}}\ \text{for all }\A\in\SS^3(\Rn)
\]
with the corresponding inner product defined as
\[
\langle \A,\B\rangle:=\sum_{i,j,k}a_{ijk}b_{ijk}.
\]

It can be shown that
\[
\rho(\A):=\max\{\langle\A,\x^{\otimes 3}\rangle\colon \x^\mathsf{T}\x=1\}
\]
defines a norm on $\SS^3(\Rn)$ \cite{Qi:rat}.  As for the matrix case, $\rho(\A)$ is called the \textit{spectral radius} of $\A$. It can be shown that
\begin{equation*}
\rho(\A)\leq \|\A\|.
\end{equation*}

Define
\[
\nu(n,s):=\frac{n^{s+1}-1}{n-1}\quad \text{and}\quad
\zeta(n,s):={n+s-1\choose n-1}.
\]
Further basic notations are put in Appendix~\ref{sec:indices}, such as several index sets, moment sequences and matrices, extended moment sequences and matrices, and localizations, etc.

\subsection{Identifying tensors with extended moment matrices}\label{sec:tensor-matrix}
Given a third order symmetric tensor $\A\in\operatorname{S}^3(\mathbb R^{n})$, we will identify it with a block of an extended moment matrix $M^{\A}\in\operatorname{S}^2(\mathbb R^{\nu(n,2)})$ via
\[
m^{\A}_{\alpha,\beta}:=\begin{cases}  a_{\alpha+\beta}&\ \text{if }|\alpha+\beta|=3,\\ 0&\text{otherwise},\end{cases}\ \text{for all }\alpha,\beta\in\I^{\leq 2}.
\]
Note that index sets such as $\I^{\leq 2}$ are given in Appendix~\ref{sec:matricization}.

Thus, $M^{\A}$ is a structured matrix with the following partitioned blocks
\begin{equation}\label{eq:matrix-structure}
M^{\A}=\begin{bmatrix}0&0&0\\ 0&0&M_{12}\\ 0&M_{21}&0\end{bmatrix},
\end{equation}
where $M_{ij}\in\mathbb R^{n^i\times n^j}$ is the block matrix of the extended moment matrix $M^\A$ corresponding to the block for the monomials $\x^{\alpha+\beta}$ with $\alpha\in \I^i$ and $\beta\in \I^j$. For the sake of easy reference, we will reserve the notation
\begin{equation}\label{eq:matrix-tensor}
M(\A):=M_{12}=M_{21}^\mathsf{T}\in\mathbb R^{n\times n^2}.
\end{equation}

$M(\A)$ is different from the Catalecticant matrix of $\A$ \cite{L-12}.
In view of the identification \eqref{eq:matrix-tensor}, we will interchangeably refer to a given tensor by $\A$ and $M(\A)$.

\begin{proposition}\label{prop:inner}
It holds
\begin{equation*}
\|\A\|^2= \|M(\A)\|^2.
\end{equation*}
\end{proposition}
\subsection{Polynomial Identification}\label{sec:polynomial}
The problem \eqref{eq:best-r} can be parameterized as
\begin{equation}\label{eq:best-r-p}
\min\bigg\{\frac{1}{2}\big\|\mathcal A-\sum_{i=1}^r\lambda_i\mathbf x_i^{\otimes 3}\big\|^2 : \lambda_i\geq 0,\ \|\mathbf x_i\|=1\ \text{for all }i=1,\dots,r\bigg\}.
\end{equation}
Note that the constraint $\|\mathbf x_i\|=1$ is added to remove some ambiguity, since the unitary scaling between $\mathbf x_i$'s and $\lambda_i$ gives the same approximate tensor. Problem~\eqref{eq:best-r-p} can be studied by applying techniques directly from polynomial optimization \cite{L-09}. However, this may not be the best way from the computational perspective. Take the case $n=10$ and $r=2$ for example, the standard relaxation will give an SDP with matrix size around $\zeta(25,4)=12650$ and number of equations around $\zeta(25,4)^2/2$ \cite{L-09}. The current SDP solvers have limited ability to solve such instances \cite{SDM,SDPT3,ZST-10,YST-15}, let alone certifying global optimality of \eqref{eq:best-r}. In this paper, we will present a method which gives a nonlinear matrix optimization problem with matrix size around $\zeta(n+1,2)=\zeta(11,2)=55$.

In the following, for simplicity, we will say $B:=M(\B)$ is a feasible (an optimal) solution of \eqref{eq:best-r} in view of the equivalence in Section~\ref{sec:tensor-matrix}.

For easy reference, in the following, we will denote the $n-1$-dimensional sphere in $\mathbb R^n$ as $\mathbb S^{n-1}$, i.e.,
\[
\mathbb S^{n-1}:=\{\mathbf x\in\mathbb R^n\colon \mathbf x^\mathsf{T}\mathbf x=1\}.
\]
Let $\mathcal M(\mathbb S^{n-1})$ be the set of Borel measures on $\mathbb S^{n-1}$. The \textit{support} of a Borel measure $\mu\in \mathcal M(\mathbb S^{n-1})$ is denoted as $\operatorname{supp}(\mu)$, which is defined as the smallest closed set $S\subseteq\mathbb S^{n-1}$ such that $\mu(\mathbb S^{n-1}\setminus S)=0$.
Given a measure $\mu\in\mathcal M(\mathbb S^{n-1})$, we denote by $\|\mu\|_0$ the cardinality of its support $\operatorname{supp}(\mu)$.  Whenever $\mu$ is finitely supported, $\|\mu\|_0$ counts the number of points in the support; otherwise $\|\mu\|_0:=+\infty$.
Let $\x\in\mathbb S^{n-1}$, then $\delta_{\x}\in \mathcal M(\mathbb S^{n-1})$ is the \textit{Dirac measure} at $\x$, with support $\{\x\}$ and having mass $1$ at $\x$ and mass $0$ elsewhere. If a measure $\mu\in\mathcal M(\mathbb S^{n-1})$) has finite support, whose cardinality is $r\geq 0$, then it can be represented as
\begin{equation}\label{eq:finite}
\mu=\sum_{i=1}^r\lambda_i\delta_{\mathbf x_i}
\end{equation}
for some $\mathbf x_i\in\mathbb S^{n-1}$ and $\lambda_i>0$ with $i=1,\dots,r$. In this case, it is called an \textit{$r$-atomic} measure.

A given third order symmetric tensor $\mathcal A$ can be uniquely decoded via
\begin{equation}\label{eq:entries-coding}
a_{i_1i_2 i_3}\leftrightarrow a_{\alpha}\ \text{with }\alpha=(\alpha_1,\dots,\alpha_n)\in\mathbb N^n_{=3}\ \text{via }\mathbf x^\alpha:=\prod_{i=1}^nx_i^{\alpha_i}=x_{i_1}x_{i_2} x_{i_3}.
\end{equation}
It is easy to see that this correspondence is one to one. With this correspondence, a third order symmetric tensor $\mathcal A$ can be interpreted as a \textit{truncated extended moment sequence} (abbreviated as \textit{tems}), which is defined as a vector
\begin{equation*}
(a_{\alpha})_{\colon \alpha\in \I^3}\in\mathbb R^{n^3}.
\end{equation*}

More precisely, a third order symmetric tensor $\mathcal A$ with a rank-$r$ decomposition as \eqref{eq:non-decomp} can be naturally restated as a truncated extended moment sequence of a finite ($r$-atomic) Borel measure \eqref{eq:finite} on $\mathbb S^{n-1}$. Actually, with $\mu=\sum_{i=1}^r\lambda_i\delta_{\mathbf x_i}$, it follows from \eqref{eq:non-decomp} and \eqref{eq:entries-coding} that
\[
a_{i_1i_2i_3}=a_{\alpha}=\sum_{i=1}^r\lambda_i (\x_i)_{i_1}(\x_i)_{i_2}(\x_i)_{i_3}=\sum_{i=1}^r\lambda_i \mathbf x_i^\alpha=\int_{\mathbb S^{n-1}}\mathbf x^\alpha \operatorname{d}\mu (\mathbf x)
\]
for all $i_1,i_2,i_3\in\{1,\dots,n\}$ and the corresponding $\alpha$ such that $\prod_{i=1}^nx_i^{\alpha_i}=x_{i_1} x_{i_2} x_{i_3}$.
More concisely, \eqref{eq:non-decomp} can be written as
\begin{equation}\label{eq:moment-rep}
\mathcal A=\sum_{i=1}^r\lambda_i\mathbf x^{\otimes 3}_i\simeq \int_{\mathbb S^{n-1}}\mathbf x^{\otimes 3}
\operatorname{d}\mu (\mathbf x),
\end{equation}
where the symbol ``$\simeq$" is understood as the obvious correspondence between the tensor $\A$ and the vector on the right hand side. Therefore, from this perspective, each third order symmetric tensor can be regarded as a truncated extended moment sequence of total degree three and vice versa.

With the moment representation \eqref{eq:moment-rep}, the third order symmetric tensor
best rank-$r$  approximation problem \eqref{eq:best-r} can be reformulated as a quadratic moment optimization problem with support constraint as follows.
\begin{proposition}[Moment Reformulation]\label{prop:moment-reformulate}
For any given third order symmetric tensor $\mathcal A$ and any nonnegative integer $r$, the best rank-$r$ tensor approximation problem \eqref{eq:best-r} is equivalent to the following moment optimization problem
\begin{equation}\label{eq:non-opt-rank}
\begin{array}{rl}
\min& \frac{1}{2}\|\mathcal A-\mathcal B\|^2\\[3pt]
\text{s.t.}& \mathcal B\simeq\int_{\mathbb S^{n-1}}\mathbf x^{\otimes 3}
\operatorname{d}\mu (\mathbf x),
\\
&\|\mu\|_0\leq r,\\
&\mu\in\mathcal M(\mathbb S^{n-1})
\end{array}
\end{equation}
in the sense that $(\mathcal B=\sum_{i=1}^r\lambda_i\mathbf x^{\otimes 3}_i, \mu=\sum_{i=1}^r\lambda_i\delta_{\mathbf x_i})$ is an optimal solution of \eqref{eq:non-opt-rank} whenever $\sum_{i=1}^r\lambda_i\mathbf x^{\otimes 3}_i$ forms an optimal solution of \eqref{eq:best-r} and vice verse.
\end{proposition}

\begin{proof}
It follows from the preceding discussions.
\end{proof}

Note that if $r$ is chosen as $\rk(\A)$, then \eqref{eq:non-opt-rank} becomes the tensor rank decomposition problem. While $\rk(\A)$ is difficult to find \cite{Has:ran}, an upper bound $r$ is usually given, which then makes \eqref{eq:non-opt-rank} a tensor decomposition problem.

There are two difficult issues in solving \eqref{eq:non-opt-rank}. The first one is the constraint $\|\mu\|_0\leq r$, and the other one is the characterization for the set $\mathcal M(\mathbb S^{n-1})$.
For the latter, there are standard positive semidefinite relaxation schemes for approximating the set $\mathcal M(\mathbb S^{n-1})$ exteriorly \cite{L-01,L-09,nie2023moment,N-14-ATKMP,N-15}. For the former, we will develop a dual certification technique from optimization. Combing these two techniques, we will present a possibility, for the first time, to certify the global optimality of the third order symmetric tensor best rank-$r$ approximation problem \eqref{eq:best-r}.

\subsection{Semidefinite Relaxation for the Measure}\label{sec:sdp}
A positive semidefinite symmetric matrix $A\in\SS^2(\mathbb R^n)$ is written as $A\succeq 0$ or $A\in\SS^n_+$.
Let $k$ be a positive integer, and $\y\in\mathbb R^{\zeta(n+1,2k)}$ be the \textit{moment sequence} of a given measure $\mu\in\mathcal M(\mathbb S^{n-1})$ up to degree $2k$, i.e.,
\begin{equation*}
\y=\int_{\mathbb S^{n-1}}\x^{2k}
\operatorname{d}\mu (\mathbf x).
\end{equation*}
It is well-known that if the measure is finitely supported with $\|\mu\|_0=r$, then the rank of the truncated moment matrix $\operatorname{rank}(\mathcal M_k(\mathbf y))\leq r$ for every positive integer $k$.
More details on moment sequences and matrices are presented in Appendix~\ref{app:moment}.
We refer the reader to \cite{L-09} for more basic notions and concepts on semidefinite relaxation hierarchy of polynomial optimization.

With the observation on the rank constraint, we consider the following problem:
\begin{equation}\label{eq:non-opt-rank-relax}
\begin{array}{rl}
\min& \frac{1}{2}\|\mathcal A-\mathcal B\|^2\\[3pt]
\text{s.t.}& \mathcal B\simeq\y|_{\mathbb N^n_{=3}},\\
&\operatorname{rank}(\mathcal M_k(\mathbf y))\leq r,\\
&\mathcal M_k(\mathbf y)\succeq 0,\\
& \LL_k(\mathbf y)=0,\ \mathbf y\in\mathbb R^{\zeta(n+1,2k)},
\end{array}
\end{equation}
where $k\geq 2$, $\mathcal M_k(\mathbf y)\in\operatorname{S}^2(\mathbb R^{\zeta(n+1,k)})$ represents the $k$-th moment matrix of the moment sequence $\mathbf{y}$, $\LL_k(\mathbf y)\in\operatorname{S}^{\zeta(n+1,k-1)}$ represents the $(k-1)$-th localizing matrix of $1-\x^\mathsf{T}\x$ at $\mathbf{y}$.

The next result is a basis for the subsequent analysis.

\begin{proposition}\label{prop:exact}
Let $k>1$ and $\mathbf y^*\in \mathbb R^{\zeta(n+1,2k)}$ be an optimal solution of \eqref{eq:non-opt-rank-relax}. If the $k$-th flatness condition holds (cf.\ Appendix~\ref{sec:flatness}), i.e.,
\begin{equation}\label{eq:flatness-rank}
\operatorname{rank}(\mathcal M_k(\mathbf y^*))=\operatorname{rank}(\mathcal M_{k-1}(\mathbf y^*)),
\end{equation}
then $\mathbf y^*$ is the truncated moment sequence of a unique $\operatorname{rank}(\mathcal M_k(\mathbf y^*))$-atomic measure and $\B$ is a best rank-$r$ approximation of the given tensor $\mathcal A$, i.e., a global minimizer of \eqref{eq:best-r}.
\end{proposition}

\begin{proof}
First of all, problem \eqref{eq:non-opt-rank-relax} is a relaxation of the problem \eqref{eq:non-opt-rank}. Thus, if there is an optimal solution $\B$ of \eqref{eq:non-opt-rank-relax} such that it is a feasible solution of \eqref{eq:non-opt-rank}, then it must be an optimal solution of the problem \eqref{eq:non-opt-rank}.

Since the sequence $\mathbf y^*$ satisfies the flatness condition, by a well-known result of Curto and Fialkow \cite{Curto-Fialkow:tkm}, it follows that the sequence $\mathbf y^*$ admits a unique measure supported by $\mathbb S^{n-1}$, which is $\operatorname{rank}(\mathcal M_k(\mathbf y^*))$-atomic. This, together with the rank constraint $\operatorname{rank}(\mathcal M_k(\mathbf y))\leq r$, implies that $\mathcal B^*\simeq\mathbf y^*|_{\mathbb N^n_{=3}}$ is a tensor with rank at most $r$, which is thus a feasible solution of \eqref{eq:non-opt-rank} with exactly the same objective function value as that of \eqref{eq:non-opt-rank-relax}.
\end{proof}


There is a method to extract the support of a measure whenever the flatness is satisfied \cite{HL-05,L-09}. Thus, an optimal solution for \eqref{eq:best-r} which is a tensor of rank at most $r$ can be computed, if an optimal solution of \eqref{eq:non-opt-rank-relax} with the flatness condition \eqref{eq:flatness-rank} being satisfied. In our numerical computation, the method in \cite{HL-05} is adopted.
\section{Rank Constrained Matrix Optimization}\label{sec:matrix}
The problem~\eqref{eq:non-opt-rank-relax} is a nonconvex optimization problem with rank constraint, which is NP-hard in general. Nonetheless, a much harder part is \blue{to} certify the global optimality for a candidate of \eqref{eq:non-opt-rank-relax}. Thus, Proposition~\ref{prop:exact} can merely be utilized in few peculiar scenarios.
In order to utilize a dual certificate for global optimality for a wider class of problems, we will propose a carefully designed variation for it.

\subsection{Reformulation}\label{sec:reform}
We consider the following optimization problem
\begin{equation}\label{eq:non-opt-rank-reform}
\begin{array}{rl}
\min& \psi(B,X):=  \frac{1}{2}\|M(\A)-B\|^2+\sigma\langle E_0,X\rangle
\\[3pt]
\text{s.t.}
&B-\mathcal P_k(\y)=0,\\
& X-\mathcal M_k(\y)=0,\\
& \mathcal L_k(\y)=0,\\
&\operatorname{rank}(B)\leq r,\\
&X\succeq 0,
\end{array}
\end{equation}
where $\sigma\geq 0$ is a \textit{control parameter}, $\mathcal P_k(\y)$ is the projection onto the block sub-matrix of the $k$-th extended moment matrix $\mathcal G_k(\y)$ generated by $\y$ corresponding to the $M(\A)$ block (cf.\ \eqref{eq:matrix-structure} and \eqref{eq:matrix-tensor}),  and $E_0$ is the symmetric positive semidefinite diagonal matrix with the upper left
most  element being one and the others being zeros.

Compared with \eqref{eq:non-opt-rank-relax}, both the parameter $\sigma$ and the rank constraint on $B$ instead of $\mathcal{M}_k(\y)$ are for dual certificate reasons which will be addressed later. The parameter $\sigma$ is also for numerical considerations
 (cf.\ Lemma~\ref{lem:strict}).
In general, both $\sigma>0$ and $\sigma=0$ are allowable in problem \eqref{eq:non-opt-rank-reform}.
Actually, if $\sigma=0$ is chosen, then \eqref{eq:non-opt-rank-reform} is \eqref{eq:non-opt-rank-relax} except that the constraint $\operatorname{rank}(B)\leq r$
is employed instead of the constraint $\operatorname{rank}(\mathcal M_k(\y))\leq r$. Since the matrix $B$ in \eqref{eq:non-opt-rank-reform} is a block sub-matrix of $\mathcal G_k(\y)$ (which is equivalent to $\M_k(\y)$ in the sense of Lemma~\ref{lem:moment}), problem \eqref{eq:non-opt-rank-reform} is a relaxation of \eqref{eq:non-opt-rank-relax}. Nevertheless, exact relaxation results will be shown in Theorem~\ref{thm:sub-optimal}. For general $\sigma>0$, quantified quasi-optimal approximation results will be given in Section~\ref{sec:quasi-optimal}.

Problem~\eqref{eq:non-opt-rank-reform} is called the \textit{$k$-th relaxation} of problem \eqref{eq:best-r}. Since $k\geq 2$, the second relaxation is called the \textit{basic relaxation}.
For the optimization problem \eqref{eq:non-opt-rank-reform}, $X$ and $B$ are determined once $\y$ is given. Thus, for simplicity, unless otherwise stated, only the variable $\y$ is referred when we talk about feasible or optimal solutions.
Note that the feasible set of \eqref{eq:non-opt-rank-reform} is closed and the objective function is a polynomial which is bounded from below on the feasible set.
Moreover, we can prove the following result.
\begin{proposition}[Solvability]\label{prop:level}
Each level set of the feasible set of the optimization problem \eqref{eq:non-opt-rank-reform} is bounded for $\sigma>0$, and there is an optimizer of \eqref{eq:non-opt-rank-reform} for each $\sigma\geq 0$.
\end{proposition}

\begin{proof}
If $\sigma>0$, we have that in each level set of the feasible set of problem \eqref{eq:non-opt-rank-reform}, both $B$ and $y_{\0}$ must be bounded.
It follows from the constraint $\mathcal L_k(\y)=0$ that
\[
y_{\0}=\sum_{i=1}^ny_{2\mathbf e_i}.
\]
It then follows from the constraint $X\succeq 0$ that all $y_{2\mathbf e_i}$ are nonnegative and bounded. In turn, it follows that all $y_{\alpha}$ with $|\alpha|=2$ are bounded by the positive semidefiniteness of $X$. The boundedness of $y_{\alpha}$ with $|\alpha|=1$ follows from the boundedness of the matrix $B$ and the constraint $\mathcal L_k(\y)=0$. In the following, we show that each $y_{2\mathbf e_i+2\mathbf e_j}$ is bounded, which will imply the boundedness of all $y_{\alpha}$ with $|\alpha|=4$ by the positive semidefinitenss of $X$.
Since it follows from $\mathcal L_k(\y)=0$ that
\[
y_{2\mathbf e_i}=\sum_{j=1}^ny_{2\mathbf e_i+2\mathbf e_j},
\]
the boundedness of the left hand side and the positive semidefiniteness of $X$ imply the desired result. This proves the boundedness for the case $k=2$.

The general case for $k>2$ follows from a similar argument through induction. We omit the details. The solvability for the case $\sigma>0$ then follows immediately.

The solvability for the case for $\sigma=0$ has a different argument. It is clear that the projection of the level set of the feasible set onto the $B$ part is bounded. In the following, we show that it is also closed. Then the conclusion follows.

For each given $B$ in the closure of this given projection, it corresponds to a third order symmetric tensor $\B$. Each third order symmetric tensor has a rank decomposition as \eqref{eq:non-decomp}, corresponding to a finite measure $\mu=\sum_{i=1}^r\lambda_i\x_i$. Let $\bar\y$ be the moment sequence generated by $\mu$ and $(\bar X,\bar\y,\bar B)$ the defined point by the first two constraints in \eqref{eq:non-opt-rank-reform}.
It is clear that all the constraints are satisfied. Moreover, the resulting feasible point $(\bar X,\bar\y,\bar B)$ has the same objective function value with the given $B$, since $B=\bar B$. Thus, the point $(\bar X,\bar\y,\bar B)$ is in the level set of the feasible set from which the projection is performed. Consequently, $B$ is in the projection and the projected set is closed as desired. The proof is completed.
\end{proof}

\subsection{Duality and Feasibility}\label{sec:dual-problem-lag}
In this section, we discuss the dual problem of \eqref{eq:non-opt-rank-reform}.
To that end, the projection of a given matrix onto the set of matrices of rank at most $r$ is involved. Let $\operatorname{R}(r)\subseteq\mathbb R^{m\times n}$ be the set of matrices in $\mathbb R^{m\times n}$ with rank at most $r\leq \min\{m,n\}$. We will use $\Pi_{\operatorname{R}(r)}(A)\subset\operatorname{R}(r)$ to denote the optimal solution set for the problem
\[
\min_X\ \frac{1}{2}\|A-X\|^2\ \ \text{s.t. }\rk(X)\leq r.
\]
$\Pi_{\operatorname{R}(r)}(A)$ can be a set with infinitely many elements, but each element $X\in \Pi_{\operatorname{R}(r)}(A)$ has the same norm $\|X\|$. Thus, $\|\Pi_{\operatorname{R}(r)}(A)\|$ can be used to define this common constant.
We refer the reader to Appendix~\ref{app:low-rank} for more details and the necessary nonsmooth analysis for this projection.

\begin{proposition}[Lagrangian Dual Problem]\label{prop:lagrangian-problem}
The Lagrangian dual problem of \eqref{eq:non-opt-rank-reform} is
\begin{equation}\label{eq:lagrange-dual-min}
\begin{array}{rl}
\min& \frac{1}{2}\|\Pi_{\operatorname{R}(r)}(M(\A)-U)\|^2\\[3pt]
\operatorname{s.t.}& \mathcal M_k^*(Z)+\mathcal P_k^*(U)-\mathcal L_k^*(W)=\sigma\mathcal M_k^*(E_0),\\
 &Z\succeq 0.
\end{array}
\end{equation}
\end{proposition}

\begin{proof}
The Lagrangian function of problem \eqref{eq:non-opt-rank-reform} is
\begin{multline*}
L(B,X,\y;U,V,W):=\frac{1}{2}\|M(\A)-B\|^2+\sigma\langle E_0,X\rangle\\
+\langle U,B-\mathcal P_k(\y)\rangle
+\langle V,X-\mathcal M_k(\y)\rangle
+\langle W,\mathcal L_k(\y)\rangle,
\end{multline*}
and problem \eqref{eq:non-opt-rank-reform} can be equivalently written as
\[
\min_{\operatorname{rank}(B)\leq r,\ X\succeq 0,\ \y}\ \ \quad\ \ \max_{U,V,W}L(B,X,\y;U,V,W).
\]

Let
\begin{equation*}
g(U,V,W):=\min_{X\succeq 0,\ \operatorname{rank}(B)\leq r,\ \y}L(B,X,\y;U,V,W).
\end{equation*}
The Lagrangian dual problem of \eqref{eq:non-opt-rank-reform} is then (cf.\ \cite{B-99})
\begin{equation*}
\max_{U,V,W} g(U,V,W).
\end{equation*}

By a direct calculation, we have
\begin{align*}
&\min_{X\succeq 0,\ \operatorname{rank}(B)\leq r,\ \y}L(B,X,\y;U,V,W)\\
=&\min_{\operatorname{rank}(B)\leq r}\big\{\frac{1}{2}\|M(\A)-B\|^2
+\langle U,B\rangle\big\}\\
&+\min_{X\succeq 0}\big\{\langle \sigma E_0+V, X\rangle\big\}
+\min_{\y}\big\{\langle W,\mathcal L_k(\y)\rangle-\langle V,\mathcal M_k(\y)\rangle-\langle U,\Pp_k(\y)\rangle\big\}\\
=&\min_{\operatorname{rank}(B)\leq r}\big\{\frac{1}{2}\big(\|B-(M(\A)-U)\|^2-\|M(\A)-U\|^2+\|M(\A)\|^2
\big)\big\}\\
&+\min_{X\succeq 0}\big\{\langle \sigma E_0+ V,X\rangle\big\}-\delta_{\{\mathbf 0\}}(\mathcal L_k^*(W)-\mathcal M_k^*(V)-\Pp_k^*(U))\\
=&\frac{1}{2}\big(\|M(\A)\|^2 - \|\Pi_{\operatorname{R}(r)}(M(\A)-U)\|^2\big)\\
&-\delta_{\operatorname{S}^n_+}( \sigma E_0+V)-\delta_{\{\mathbf 0\}}(\mathcal L_k^*(W)-\mathcal M_k^*(V)-\Pp_k^*(U)),
\end{align*}
where $\mathcal L_k^*, \mathcal M_k^*$ and $\Pp_k^*$ are the adjoint operators of $\mathcal L_k, \mathcal M_k$ and $\Pp_k$ respectively.
Therefore,
we have the dual problem
\begin{equation*}
\begin{array}{rl}
\max&\frac{1}{2}\|M(\A)\|^2 -\frac{1}{2}\|\Pi_{\operatorname{R}(r)}(M(\A)-U)\|^2
\\[3pt]
\text{s.t.}&  \mathcal L_k^*(W)-\mathcal M_k^*(V)-\mathcal P_k^*(U)=\mathbf 0,\\
 &\sigma E_0+V\succeq 0.
\end{array}
\end{equation*}
In a more concise form, it is
\begin{equation}\label{eq:lagrange-dual-compact}
\begin{array}{rl}
\max& \phi(U):=\frac{1}{2}\|M(\A)\|^2 -\frac{1}{2}\|\Pi_{\operatorname{R}(r)}(M(\A)-U)\|^2 \\[3pt]
\text{s.t.}&  \mathcal M_k^*(Z)+\mathcal P_k^*(U)-\mathcal L_k^*(W)=\sigma\mathcal M_k^*(E_0),\\
 &Z\succeq 0.
\end{array}
\end{equation}
We see that \eqref{eq:lagrange-dual-min} is actually the minimization formulation of \eqref{eq:lagrange-dual-compact}.
\end{proof}

Problem \eqref{eq:lagrange-dual-compact} (equivalently \eqref{eq:lagrange-dual-min}) is a convex optimization problem, as expected, but with a nonsmooth objective function
$\phi(U)$.

One advantage of the formulation \eqref{eq:lagrange-dual-min} is that we can interpret the constraint via sums of squares of polynomials.
\begin{lemma}[Dual Feasibility]\label{lem:feasibility}
A triplet $(U,W,Z)$ is a feasible solution of problem \eqref{eq:lagrange-dual-min} if and only if there exist a homogeneous polynomial $u(\x)$, and polynomials $w(\x)$ and $z(\x)$ with
\[
\operatorname{deg}(w(\x))\leq 2k-2,\ \operatorname{deg}(u(\x))=3,\ \text{and }\operatorname{deg}(z(\x))\leq 2k
\]
such that
\begin{equation}\label{eq:sum-of-squares}
z(\x)=(1-\x^\mathsf{T}\x)w(\x)-u(\x)+\sigma
\end{equation}
is a sum of squares of polynomials. Moreover, there is a correspondence between the triplets $(U,W,Z)$ and $(u(\x), w(\x), z(\x))$ as indicated in Appendix~\ref{sec:moment-matrix}.
\end{lemma}

\begin{proof}
Recall the sizes of the triplet $(U,W,Z)$, which are
\[
U\in\mathbb R^{n\times n^2},\ W\in \SS^2(\mathbb R^{\zeta(n+1,k-1)})\ \text{and }Z\in\SS^2(\mathbb R^{\zeta(n+1,k)}).
\]
Let $\x^k$ be the monomial basis up to order $k$ defined as in \eqref{eq:monomial}
and let $\x^{\otimes 2}$ be the extended monomial basis of order $2$ defined as in \eqref{eq:tensor-1}. Let
\[
w(\x):=(\x^{k-1})^\mathsf{T}W\x^{k-1},\ u(\x):=\x^\mathsf{T}U\x^{\otimes 2},\ \text{and }z(\x):=(\x^k)^\mathsf{T}Z\x^k.
\]
Then by the feasibility of $(U,W,Z)$, we have that
\[
(1-\x^\mathsf{T}\x)w(\x)+\sigma-u(\x)-z(\x)=0.
\]
Since $Z$ is a positive semidefinite matrix, which is equivalent to having
the polynomial $z(\x)$ being a sum of squares of polynomials \cite{L-09}, the conclusion follows.
\end{proof}

\begin{lemma}[Strict Feasibility]\label{lem:strict}
For any positive $\sigma>0$ and integer $k\geq 2$, problem \eqref{eq:lagrange-dual-min} is strictly feasible, i.e., there exists a triplet $(U,W,Z)$ with $Z\succ 0$ such that $\mathcal M_k^*(Z)+\mathcal P_k^*(U)-\mathcal L_k^*(W)=\sigma\mathcal M_k^*(E_0)$.
\end{lemma}

\begin{proof}
First note that
for any given positive scalars $\mu_i$ ($i=0,\dots,k$), we can find a  positive definite
diagonal matrix $A$ such that
\[
(\x^k)^\mathsf{T}A\x^k=\sum_{i=0}^k\mu_i(\x^\mathsf{T}\x)^i.
\]
Thus, the conclusion will follow, by Lemma~\ref{lem:feasibility}, if we can find polynomials $w(\x)$ and $u(\x)$ such that
\[
(1-\x^\mathsf{T}\x)w(\x)+\sigma-u(\x)=\sum_{i=0}^k\mu_i(\x^\mathsf{T}\x)^i
\]
for positive $\mu_i$'s.
This can be fulfilled by taking $u(\x)=0$ and
\[
w(\x) := \sum_{i=0}^{k-1}\lambda_i(\x^\mathsf{T}\x)^i
\]
for any choices of $\lambda_i$'s such that
$
-\sigma<\lambda_0<\lambda_1<\dots<\lambda_{k-1}<0.
$
In this case
\[
\mu_0=\sigma+\lambda_0>0,\ \mu_i=\lambda_i-\lambda_{i-1}>0\ \text{for }i=1,\dots,k-1\ \text{and }\mu_k=-\lambda_{k-1}>0.
\]
This completes the proof.
\end{proof}

The strict feasibility does not hold for $\sigma=0$.
The cubic form $u(\x)$ either is zero or takes negative value on the sphere.
Thus, $z(\x)=(1-\x^\mathsf{T}\x)w(\x)-u(\x)$ either is identically zero or takes negative function value on the sphere.
Consequently, if $\sigma=0$, the feasibility condition \eqref{eq:sum-of-squares} forces $U=0$, and hence problem \eqref{eq:lagrange-dual-min} becomes a feasibility problem with constant objective function. Then the possibility for strong duality between \eqref{eq:non-opt-rank-reform} and \eqref{eq:lagrange-dual-compact} is weakened.
As a result, it is necessary for numerical reasons to impose positive $\sigma$ and therefore the relationship between the optimal solutions for \eqref{eq:non-opt-rank-reform} and those for the original best approximation problem \eqref{eq:best-r} should be established.

\subsection{Optimality}\label{sec:optimality}
The following conclusion is classical, which follows from the saddle point theorem \cite{B-99}.
\begin{proposition}[Lagrangian Duality]\label{prop:lag-duality}
Let $(B,X,\y)$ and $(U,V,W)$ be feasible solutions of problems~\eqref{eq:non-opt-rank-reform} and \eqref{eq:lagrange-dual-min} respectively. Then we have
\begin{equation*}
\psi(B,X)\geq\phi(U).
\end{equation*}
If $\psi(B,X)=\phi(U)$, then both $(B,X,\y)$ and $(U,V,W)$ are optimal solutions of problems~\eqref{eq:non-opt-rank-reform} and \eqref{eq:lagrange-dual-min} respectively.
\end{proposition}

While the primal problem \eqref{eq:non-opt-rank-reform} is a nonlinear semidefinite matrix optimization problem which is nonconvex due to the rank constraint, the dual problem \eqref{eq:lagrange-dual-min} is a nonlinear convex semidefinite matrix optimization problem. By the feasibility characterization in Lemma~\ref{lem:feasibility}, the optimality of \eqref{eq:lagrange-dual-min} can be concisely determined with the help of convex analysis \cite{R-70}.
\begin{proposition}[Optimality]\label{prop:solution}
We have that a feasible solution $(\bar U,\bar W, \bar Z)$ of \eqref{eq:lagrange-dual-min} is an optimal solution if there exists a vector $\bar \y\in\mathbb R^{\zeta(n+1,2k)}$ such that
\begin{eqnarray}\label{eq:optimality}
\begin{array}{l}
\mathcal L_k(\bar \y)=0,\ \mathcal M_k(\bar \y)\succeq 0,\ \langle\bar Z,\mathcal M_k(\bar \y)\rangle=0,\
\text{and }\\[3pt]
\mathcal P_k(\bar \y)\in\operatorname{conv}\big(\Pi_{\operatorname{R}(r)}(M(\A)-\bar U)\big).
\end{array}
\end{eqnarray}
It becomes also a necessary condition if $\sigma>0$.
\end{proposition}

\begin{proof}
This follows from Lemma~\ref{lem:strict}, Lemma~\ref{lem:subdiff} in Appendix B and standard convex analysis \cite[Theorem~27.4]{R-70}.

The sufficiency is important for our subsequent analysis and we give
a proof by the following direct calculation. For any feasible $(U,W,Z)$ of \eqref{eq:lagrange-dual-min}, we have
\begin{align*}
\frac{1}{2}\|\Pi_{\operatorname{R}(r)}(M(\A)-U)\|^2&\geq \frac{1}{2}\|\Pi_{\operatorname{R}(r)}(M(\A)-\bar U)\|^2+\langle -\mathcal P_k(\bar \y),U-\bar U\rangle\\
&=\frac{1}{2}\|\Pi_{\operatorname{R}(r)}(M(\A)-\bar U)\|^2-\langle \bar \y,\mathcal P_k^*(U)-\mathcal P_k^*(\bar U)\rangle\\
&=\frac{1}{2}\|\Pi_{\operatorname{R}(r)}(M(\A)-\bar U)\|^2-\langle \bar \y,\mathcal L_k^*(W)-\mathcal L_k^*(\bar W)\rangle\\
&\ \ \ \ \ \ \ \ \ \ +\langle \bar \y,\mathcal M_k^*(Z)-\mathcal M_k^*(\bar Z)\rangle\\
&=\frac{1}{2}\|\Pi_{\operatorname{R}(r)}(M(\A)-\bar U)\|^2+\langle \mathcal M_k(\bar \y),Z-\bar Z\rangle\\
&=\frac{1}{2}\|\Pi_{\operatorname{R}(r)}(M(\A)-\bar U)\|^2+\langle \mathcal M_k(\bar \y),Z\rangle\\
&\geq \frac{1}{2}\|\Pi_{\operatorname{R}(r)}(M(\A)-\bar U)\|^2,
\end{align*}
where the first inequality follows from Lemma~\ref{lem:subdiff} and \eqref{eq:optimality}, the second equality from the feasibility, and the rest all follow from \eqref{eq:optimality}.
\end{proof}

We are in the position to present one of our main results.
\begin{theorem}[Dual Certification]\label{thm:optimality}
If there exists a triplet $(\bar U,\bar W, \bar Z)$ such that the feasibility of problem \eqref{eq:lagrange-dual-min} is satisfied and a vector $\bar \y$ such that the optimality condition \eqref{eq:optimality} is satisfied, and
\begin{equation}\label{eq:rank-optimality}
\operatorname{rank}(\mathcal P_k(\bar \y))\leq r,
\end{equation}
then $\bar \y$ gives an optimal solution to problem \eqref{eq:non-opt-rank-reform}.
\end{theorem}

\begin{proof}
Let
\[
\bar B:=\mathcal P_k(\bar \y)\ \text{and }\bar X:=\mathcal M_k(\bar\y).
\]
Then, $(\bar B,\bar X,\bar y)$ is a feasible solution for \eqref{eq:non-opt-rank-reform} by \eqref{eq:optimality} and \eqref{eq:rank-optimality}.
Moreover, $\bar{B}\in\Pi_{\operatorname{R}(r)}(M(\A)-\bar U)$ by Lemma~\ref{lem:low-rank} and \eqref{eq:rank-optimality}. Thus,
\begin{align*}
\phi(\bar U)=&\frac{1}{2}\|M(\A)\|^2 -\frac{1}{2}\|\Pi_{\operatorname{R}(r)}(M(\A)-\bar U)\|^2 \\
=&\frac{1}{2}\|M(\A)\|^2 +\frac{1}{2}\|M(\A)-\bar U-\Pi_{\operatorname{R}(r)}(M(\A)-\bar U)\|^2 -\frac{1}{2}\|M(\A)-\bar U\|^2\\
=&\frac{1}{2}\|M(\A)\|^2 +\frac{1}{2}\|M(\A)-\bar U-\bar B\|^2 -\frac{1}{2}\|M(\A)-\bar U\|^2\\
=&\frac{1}{2}\|M(\A)-\bar B\|^2+\langle\bar U,\bar B\rangle.
\end{align*}
Therefore, we have
\begin{align*}
\phi(\bar U)&=\frac{1}{2}\|M(\A)-\bar B\|^2+\langle\bar U,\bar B\rangle\\
&=\frac{1}{2}\|M(\A)-\bar B\|^2+\langle\bar U,\mathcal P_k(\bar \y)\rangle\\
&=\frac{1}{2}\|M(\A)-\bar B\|^2+\langle\mathcal P_k^*(\bar U), \bar \y\rangle\\
&=\frac{1}{2}\|M(\A)-\bar B\|^2+\langle\mathcal L^*_k(\bar W)-\mathcal M_k^*(\bar Z)+\sigma\mathcal M_k^*(E_0), \bar \y\rangle\\
&=\frac{1}{2}\|M(\A)-\bar B\|^2+\langle\bar W, \mathcal L_k(\bar \y)\rangle-\langle \bar Z,\mathcal M_k(\bar\y)\rangle+\sigma\langle E_0,\mathcal M_k(\bar \y)\rangle\\
&=\frac{1}{2}\|M(\A)-\bar B\|^2+\sigma\langle E_0,\bar X\rangle
\end{align*}
where the fourth equality follows from the feasibility of \eqref{eq:lagrange-dual-min} and the last equality follows from \eqref{eq:optimality}. Thus, by \eqref{eq:non-opt-rank-reform},
\[
\psi(\bar B,\bar X)=\phi(\bar U).
\]
By Lagrangian duality (cf.\ Proposition~\ref{prop:lag-duality}), we get the conclusion.
\end{proof}

\section{Quality of Approximation}\label{sec:app-quality}
In Theorem~\ref{thm:optimality}, we established a certification for global optimality of problem \eqref{eq:non-opt-rank-reform}. In this section, we will give the relationships between
the global optimal solutions of \eqref{eq:non-opt-rank-reform} and the original best low rank approximation problem \eqref{eq:best-r}. The discussions are divided into two subsections
based on the control parameter $\sigma$: the case for $\sigma=0$ and the case for $\sigma>0$.

\subsection{Best Low Rank Approximation}\label{sec:app-quality-low-rank}
In this subsection, we consider the case $\sigma=0$ in problem \eqref{eq:non-opt-rank-reform}, which is related to the best low rank approximation.

Let $\mathbb O(n)\subset\mathbb R^{n\times n}$ be the group of orthogonal matrices and $Q\in\mathbb O(n)$.
For a given $\mathcal A\in\operatorname{S}^3(\mathbb R^n)$ with $\mathcal{A}=(a_{ijk})$,  the \textit{matrix-tensor multiplication} $(Q,Q,Q)\cdot\mathcal A$  is defined as a tensor in $\operatorname{S}^3(\mathbb R^n)$ with its components being
\begin{equation}\label{eq:matrix-tensor-mt}
\big[(Q,Q,Q)\cdot\mathcal A\big]_{rst}=\sum_{i,j,k=1}^nq_{ri}q_{sj}q_{tk}a_{ijk}\ \text{for all }r,s,t\in\{1,\dots,n\}.
\end{equation}
This multiplication can be defined for general cases in an obvious way.

\begin{lemma}\label{lem:rank-r}
If $\{\x_1,\dots,\x_r\}\subset\Rn$ has rank exactly $r$, then
\[
\rk\big(\sum_{i=1}^{r+1}\lambda_i\x_i(\x_i^{\otimes 2})^\mathsf{T}\big)\leq r
\]
for any $\x_{r+1}\in\operatorname{span}\{\x_1,\dots,\x_r\}$
and $\lambda_i\geq 0$.
\end{lemma}

\begin{lemma}\label{lem:app-rank}
Let $\{\x_1,\dots,\x_s\}\subset\mathbb R^n$ be a set of nonzero vectors of rank $s-1$. Then, the matrix $B:=\sum_{i=1}^s\x_i(\x_i^{\otimes 2})^\mathsf{T}$ has rank at least $s-2$, and $B$ has rank $s-2$ if and only if $\x_i=-\x_j$ for some $i\neq j$.
\end{lemma}

\begin{proof}
Without loss of generality, we assume that $\x_1,\dots,\x_{s-1}$ are linearly independent.
There exists a nonsingular matrix $P\in\mathbb R^{n\times n}$ such that
\[
P\x_i = \e_i\ \text{for all }i=1,\dots,s-1.
\]
We have
\[
PB(P\otimes P)^\mathsf{T} =\sum_{i=1}^s(P\x_i)[(P\x_i)^{\otimes 2}]^\mathsf{T}.
\]
Thus, without loss of generality, we assume that
\[
\x_i=\e_i\ \text{for all }i=1,\dots,s-1\ \text{and }\x_s=\sum_{i=1}^{s-1}\alpha_i\e_i.
\]
Let $Q\in\mathbb R^{n^2\times n^2}$ be the permutation matrix such that
\[
\sum_{i=1}^{s-1}\e_i(\e_i^{\otimes 2})^\mathsf{T}Q = \begin{bmatrix}I_{s-1}&0\\ 0& 0\end{bmatrix}.
\]
Obviously, $\operatorname{rank}(B)=\rk(BQ)\geq s-2$. If $\operatorname{rank}(B)=s-2$, we must have that the $(s-1)\times (s-1)$ leading principal submatrix of $BQ$, which is
\[
I_{s-1}+\tilde\x_s[(\x_s^{\otimes 2})^\mathsf{T}Q]_{1:s-1},
\]
has rank $s-2$, where $\tilde\x_s=(\x_s)_{1:s-1}$. This can happen if and only if
\begin{equation}\label{eq:vector-0}
[(\x_s^{\otimes 2})^\mathsf{T}Q]_{1:s-1}=-\tilde{\x}_s^\mathsf{T}\ \text{and }\|\tilde{\x}_s\|=1.
\end{equation}
While,
\begin{equation}\label{eq:vector-1}
[(\x_s^{\otimes 2})^\mathsf{T}Q]_{1:s-1}=((\x_s)_1^2,\dots,(\x_s)_{s-1}^2).
\end{equation}
Thus, \eqref{eq:vector-0} and \eqref{eq:vector-1} imply that
\[
\x_s=-\e_i\ \text{for some }i\in\{1,\dots,s-1\}.
\]

The sufficiency is clear.
Consequently, the conclusion follows.
\end{proof}

Proposition~\ref{prop:exact} indicates that \eqref{eq:non-opt-rank-relax} is an exact relaxation of the best low rank approximation problem \eqref{eq:non-opt-rank}, the following result characterizes the approximation quality of the further relaxation \eqref{eq:non-opt-rank-reform} to \eqref{eq:non-opt-rank}.
\begin{theorem}\label{thm:sub-optimal}
Let $\y$ be an optimizer of problem \eqref{eq:non-opt-rank-reform} with $k=2$ and $\sigma=0$ satisfying $\operatorname{rank}(\mathcal M_{1}(\y))=\operatorname{rank}(\mathcal M_2(\y))$. Then
\begin{equation}\label{eq:rank-moment}
\operatorname{rank}(\y)=\operatorname{rank}(\mathcal M_2(\y))\leq r+2.
\end{equation}
Moreover,
\begin{enumerate}
\item if $\operatorname{rank}(\mathcal M_2(\y))\leq r$ or $\operatorname{rank}(\mathcal M_2(\y))=r+2$ or $r=1$, then $\Pp_2(\y)$ is an optimal solution for \eqref{eq:best-r}.

\item if $\operatorname{rank}(\mathcal M_2(\y))=r+1$, then a feasible solution $\B$ of \eqref{eq:best-r} can be constructed from $\y$ such that
\begin{equation*}
\|\A-\B\|^2\leq \|\A-\B^*\|^2 + \rho(\A-\B)^2
\end{equation*}
where $\B^*$ is an optimal solution of \eqref{eq:best-r}.
\end{enumerate}
\end{theorem}

\begin{proof}
By Proposition~\ref{prop:exact}, for the conclusion \eqref{eq:rank-moment}, it is sufficient to show that $\operatorname{rank}(\mathcal M_2(\y))\leq r+2$.

Let $\operatorname{rank}(\mathcal M_2(\y))=s$ for some integer $s\geq 0$. In the following, we consider the case $s>r$, since when $s\leq r$, the optimal $B=\Pp_2(\y)$ corresponds to a tensor of rank at most $r$. Thus, it gives an optimal solution for \eqref{eq:best-r}.

By the fact that $\y$ satisfies the second flatness condition, we have that there exist $\lambda_i>0$ and $\x_i\in\mathbb S^{n-1}$ for all $i=1,\dots,s$ such that
\[
\y=\int_{\mathbb S^{n-1}}\x^4 \operatorname{d}\mu (\x)
\]
with
$
\mu:=\sum_{i=1}^s\lambda_i\delta_{\x_i}.
$
It then follows from the definition that
\[
\mathcal M_1(\y)=\sum_{i=1}^s\lambda_i\bar \x_i\bar\x_i^\mathsf{T}
\]
with
\[
\bar\x_i:=\begin{bmatrix}1\\ \x_i\end{bmatrix}\ \text{for all }i=1,\dots,s.
\]
By the flatness condition, $\operatorname{rank}(\mathcal M_1(\y))=s>0$. Thus, the set of vectors $\{\x_1,\dots,\x_s\}$ has rank at least $s-1$.
On the other hand, by the feasibility of $\y$ for \eqref{eq:non-opt-rank-reform}, it follows that the matrix
\[
\bar B:=\sum_{i=1}^s\lambda_i\x_i(\x_i^{\otimes 2})^\mathsf{T}
\]
is of rank not greater than $r$. If the vectors $\x_1,\dots,\x_s$ are linearly independent, then the corresponding matrix $\bar B$ must have rank $s>r$, which is a contradiction. Thus, the set of vectors $\{\x_1,\dots,\x_s\}$ has rank $s-1$. Consequently, by Lemma~\ref{lem:app-rank}, we have that $r<s\leq r+2$. The conclusion \eqref{eq:rank-moment} then follows.

Moreover, by Lemma~\ref{lem:app-rank} again, we have that $s=r+2$ exactly when $\sqrt[3]{\lambda_i}\x_i=-\sqrt[3]{\lambda_j}\x_j$ for a pair $i\neq j$. But in this case, the term $\lambda_i\x_i^{\otimes 3}$ and $\lambda_j\x_j^{\otimes 3}$ combined into a zero sum. Consequently,
\[
\bar B=\sum_{k\in\{1,\dots,s\}\setminus\{i,j\}}\lambda_k\x_k(\x_k^{\otimes 2})^\mathsf{T},
\]
which corresponds to a tensor of rank $s-2=r$. Similar to the case $s=r$, an optimal solution for \eqref{eq:best-r} is found.

In the following, we consider the case $s=r+1$.

We assume, without loss of generality, that the vectors $\x_1,\dots,\x_r$ are linear independent and $\x_{r+1}\in\operatorname{span}\{\x_1,\dots,\x_r\}$. If $r=1$, then $\x_{r+1}=\x_2=\pm \x_1$ and $B=\mathcal P(\y)$ corresponds to a rank one tensor. Hence, it is an optimal solution.

By the equivalence of the matrix representation (cf.\ Section~\ref{sec:tensor-matrix}), we have that
\begin{equation}\label{eq:optimal-app}
\|\A-\sum_{i=1}^{r+1}\lambda_i\x_i^{\otimes 3}\|^2\leq \|\A-\B^*\|^2,
\end{equation}
where $\B^*$ is an optimal solution to problem \eqref{eq:best-r}, since \eqref{eq:non-opt-rank-reform} is a relaxation of \eqref{eq:best-r}.

It follows from the definition that (cf.\ \eqref{eq:matrix-tensor-mt})
\[
\|(P,P,P)\cdot\A\|=\|\A\|
\]
for an orthogonal matrix $P\in\mathbb{O}(n)$. Thus, we can assume without loss of generality that
\begin{equation}\label{eq:space}
\operatorname{span}\{\x_1,\dots,\x_r\}=\operatorname{span}\{\e_1,\dots,\e_r\},
\end{equation}
where $\e_i\in\mathbb R^n$ is the $i$-th column vector of the identity matrix of matching size for all $i\in\{1,\dots,r\}$.

On the other hand, it holds that
\[
B:=\sum_{i=1}^r\lambda_i\x_i(\x_i^{\otimes 2})^\mathsf{T}+\beta\x(\x^{\otimes 2})^\mathsf{T}
\]
has rank at most $r$ for any $\x\in\operatorname{span}\{\x_1,\dots,\x_r\}$. Therefore, it, together with the fact that $\x_{r+1}\in\operatorname{span}\{\x_1,\dots,\x_r\}$ and \eqref{eq:space}, implies that $\lambda_{r+1}({(\x_{r+1})}_{1:r})^{\otimes 3}$ is a best rank one approximation of the sub-tensor $(\A-\sum_{i=1}^{r}\lambda_i\x_i^{\otimes 3})_{{1:r}}$. Here $\mathbf u_{1:r}\in\mathbb R^r$ is the sub-vector of $\mathbf u\in\mathbb R^n$ formed by the first $r$ entries $u_1,\dots,u_r$, and $\mathcal U_{1:r}\in\SS^3(\mathbb R^r)$ is the sub-tensor of $\mathcal U\in\SS^3(\mathbb R^n)$ formed by the entries $u_{i_1i_2i_3}$ with $i_1,i_2,i_3\in\{1,\dots,r\}$.
Thus, the optimality implies that
\begin{equation}\label{eq:optimal-app-rho}
\lambda_{r+1}^2=\rho\big((\A-\sum_{i=1}^{r}\lambda_i\x_i^{\otimes 3})_{{1:r}}\big)^2\leq \rho(\A-\sum_{i=1}^{r}\lambda_i\x_i^{\otimes 3})^2
\end{equation}
and
\begin{equation}\label{eq:optimal-app-chain}
\|\A-\sum_{i=1}^{r+1}\lambda_i\x_i^{\otimes 3}\|^2=\|\A-\sum_{i=1}^{r}\lambda_i\x_i^{\otimes 3}\|^2-\lambda_{r+1}^2.
\end{equation}
Therefore, with $\B:=\sum_{i=1}^{r}\lambda_i\x_i^{\otimes 3}$, \eqref{eq:optimal-app}, \eqref{eq:optimal-app-rho} and \eqref{eq:optimal-app-chain}, we have
\[
\|\A-\B\|^2\leq \|\A-\B^*\|^2+\rho(\A-\B)^2.
\]
This completes the proof.
\end{proof}

In the sequel, we show the exact relaxation of \eqref{eq:non-opt-rank-reform} when the given tensor is orthogonally decomposable. To do this, recall that
a third order symmetric tensor $\mathcal A\in\operatorname{S}^3(\mathbb R^{n})$ is called \textit{orthogonally decomposable} (cf.\ \cite{ZG-01,K-01} \footnote{In \cite{K-01}, this notion was
referred as completely orthogonally decomposable tensors by Kolda.}) if there exist an orthonormal matrix
\[
A=[\mathbf a_{1},\dots,\mathbf a_{r}]\in\mathbb R^{n\times r}
\]
and positive numbers $\lambda_i\in\mathbb R$ for $i=1,\dots,r$ such that
\begin{equation}\label{eq:orthogonal}
\mathcal A=\sum_{i=1}^{r}\lambda_i\mathbf a_i^{\otimes 3}.
\end{equation}
The number $r$ is the rank of the tensor $\A$. It is well-known that the rank one decomposition \eqref{eq:orthogonal} of an orthogonally decomposable tensor $\A$ is unique \cite{ZG-01}.
\begin{theorem}[Exact Relaxation]\label{thm:exact-relaxtion-0}
If $\A$ is an orthogonally decomposable tensor with rank $s\leq n$, then for $\sigma=0$
and $r\leq s$, \eqref{eq:non-opt-rank-reform} is an exact relaxation of the best rank-$r$ approximation problem.
\end{theorem}

\begin{proof}
Suppose that
\[
\A=\sum_{i=1}^s\lambda_i\x_i^{\otimes 3}
\]
be an orthogonal decomposition of $\A$ with $\lambda_1\geq\dots\geq\lambda_s>0$.
Let
\[
\mu:=\sum_{i=1}^r\lambda_i\delta_{\x_i}
\]
and $\bar\y$ the corresponding moment sequence generated by the $r$-atomic measure $\mu$. It is immediate to see that $\bar\y$ is a feasible solution of problem \eqref{eq:non-opt-rank-reform}.

A dual feasible solution for \eqref{eq:lagrange-dual-compact} is $(U,V,W)=(0,0,0)$. Obviously,
\[
\operatorname{rank}(\mathcal P_k(\bar\y))\leq r,
\]
since
\[
\mathcal P_k(\bar\y)=\sum_{i=1}^r\lambda_i\x_i(\x_i^{\otimes 2})^\mathsf{T}.
\]

If $\mathcal P_k(\bar\y)\in\operatorname{conv}(\Pi_{\operatorname{R}(r)}(M(\A)))$, then by Proposition~\ref{prop:solution} and Theorem~\ref{thm:optimality}, we can conclude that $\bar\y$ is an optimal solution of problem \eqref{eq:non-opt-rank-reform}. Consequently, it is an optimizer of the best rank-$r$ approximation problem by Theorem~\ref{thm:sub-optimal}, since the flatness is satisfied.

Note that
\begin{equation}\label{eq:odt-svd}
M(\A)=\sum_{i=1}^s\lambda_i\x_i(\x_i^{\otimes 2})^\mathsf{T}
\end{equation}
and
\[
\begin{bmatrix}\x_1&\dots&\x_s\end{bmatrix}
\]
is an orthonormal matrix. Likewise, the matrix
\[
\begin{bmatrix}\x_1^{\otimes 2}&\dots&\x_s^{\otimes 2}\end{bmatrix}
\]
is also orthonormal. Thus \eqref{eq:odt-svd} is a singular value decomposition (a.k.a. SVD \cite{GV-12}) of the matrix $M(\A)$. By the classical Eckart-Young-Mirsky theorem (cf.\ \cite{HJ-85}), the truncated SVD is an optimizer of the best rank-$r$ approximation problem for the matrix $M(\A)$. Thus, $\mathcal P_k(\bar\y)\in\operatorname{conv}(\Pi_{\operatorname{R}(r)}(M(\A)))$. The result then follows.
\end{proof}

\subsection{Quasi-Optimality}\label{sec:quasi-optimal}
In this section, we discuss the case when $\sigma>0$ in problem~\eqref{eq:non-opt-rank-reform}, which is related to quasi-optimal low rank approximations of the given tensor.
\begin{definition}\label{def:quasi-optimal}
Given a nonzero tensor $\mathcal A\in\operatorname{S}^3(\mathbb R^{n})$ and a positive integer $r$, let $\overline{\mathcal B}$ be a best rank-$r$ approximation of $\mathcal{A}$ and $\alpha\geq 0$. A tensor $\mathcal{B}\in \operatorname{S}^3(\mathbb R^{n})$ is called a {\rm$\alpha$-quasi-optimal rank-$r$ approximation} of $\mathcal{A}$ if
\begin{equation*}
\|\mathcal{A}-\overline{\mathcal{B}}\|^2\leq\|\mathcal{A}-\mathcal{B}\|^2\leq \|\mathcal{A}-\overline{\mathcal{B}}\|^2+\alpha.
\end{equation*}
\end{definition}

We first show that optimal solutions of \eqref{eq:non-opt-rank-reform} can actually give best rank one approximants.
\begin{proposition}\label{prop:app-sigma}
Let $\A\in\SS^3(\mathbb R^n)$ be nonzero, $r=1$, $(B,\y,X)$ be an optimal solution of \eqref{eq:non-opt-rank-reform} with $\sigma>0$. Then, we have $B=X(1,1)\x^{\otimes 3}$
for some $\x\in \mathbb S^{n-1}$.
If $\sigma<\rho(\A)$, then $(X(1,1)+\sigma)\x^{\otimes 3}$ is a best rank one approximation of $\mathcal A$.
\end{proposition}

\begin{proof}
Note that $B=\mathcal P_k(\y)$ has rank at most one by the feasibility, then it follows that
\[
B=\lambda\x(\x^{\otimes 2})^\mathsf{T}
\]
for some unit vector $\x\in\mathbb S^{n-1}$ and $\lambda\geq 0$, since a third order symmetric tensor $\mathcal U$ has rank one if and only if its flattening matrix $M(\mathcal U)$ has rank one \cite{L-12}. First of all, we will derive from the fact that the matrix $X=\mathcal M_k(\y)$ is positive semidefinite and $\mathcal L_k(\y)=0$ that $X(1,1)\geq \lambda$.
Actually, it follows from $\mathcal L_k(\y)=0$ that
\[
\mathcal M_1(\y)=\begin{bmatrix}\beta&\lambda\x^\mathsf{T}\\ \lambda\x&A\end{bmatrix}
\]
for some positive semidefinite matrix $A\in\SS^2(\mathbb R^n)$. The case when $\lambda=0$ is trivial. In the following, we assume that $\lambda>0$ and thus $\beta>0$ by the positive semidefiniteness of $\mathcal M_1(\y)$. By Schur's complement theory \cite{HJ-85}, we have that $A\succeq \frac{\lambda^2}{\beta}\x\x^\mathsf{T}$.
While, it follows from $\mathcal L_k(\y)=0$ that
\[
\beta = \operatorname{tr}(A)\geq \frac{\lambda^2}{\beta}.
\]
Thus $X(1,1) = \beta \geq \lambda$.
The result then follows.

As a result, the optimal $X$ of \eqref{eq:non-opt-rank-reform} must have the smallest possible $X(1,1)$, which is $\lambda$. Thus, $X$ must be of rank one and $\|B\|=\lambda=X(1,1)$. Let the optimal solution be $\lambda\x^{\otimes 3}$. We must have
\begin{equation*}
\frac{1}{2}\|M(\A)-\lambda\x(\x^{\otimes 2})^\mathsf{T}\|^2+\sigma\lambda\leq \min_{\mu\geq 0,\y\in\mathbb S^{n-1}}\Big\{\frac{1}{2}\|M(\A)-\mu\y(\y^{\otimes 2})^\mathsf{T}\|^2+\sigma\mu\Big\}
\end{equation*}
by the optimality.
It also follows from the optimality of $\lambda$, which is independent of $\x$, that
\[
\lambda = \max\{\langle\A,\x^{\otimes 3}\rangle-\sigma,0\}.
\]
Since $\sigma<\rho(\A)$, there is a $\y$ such that $\A\y^3-\sigma>0$. Thus, by the global optimality, $\lambda = \langle\A,\x^{\otimes 3}\rangle-\sigma$ and
\[
\frac{1}{2}\|M(\A)-\lambda\x(\x^{\otimes 2})^\mathsf{T}\|^2+\sigma\lambda=\frac{1}{2}\|\mathcal A\|^2-\frac{1}{2}\lambda^2.
\]
It further follows from the optimality that
\[
\lambda=\langle\A,\x^{\otimes 3}\rangle-\sigma=\max\{\langle\A,\y^{\otimes 3}\rangle-\sigma\colon \y\in\mathbb S^{n-1}\}=\rho(\A)-\sigma.
\]
Hence, $(\lambda+\sigma)\x^{\otimes 3}=\rho(\A)\x^{\otimes 3}$ is a best rank one approximation of $\mathcal A$.
The conclusion then follows.
\end{proof}

The coherence of a matrix $A=[\mathbf{x}_1,\dots,\mathbf{x}_r]$, denoted by $\mu(A)$, is defined as (cf.\ \cite{lim2013blind})
\begin{equation*}
\mu(A):=\max_{i\neq j}|\langle \mathbf{x}_i,\mathbf{x}_j\rangle|.
\end{equation*}
It follows that $\mu(A)$ is the maximum absolute value of the off-diagonal elements of $A^\mathsf{T}A$.
\begin{lemma}\label{lem:kappa}
Let $A=[\mathbf{x}_1,\dots,\mathbf{x}_r]$ be a given matrix with unit columns, and $C=A^\mathsf{T}A\circ A^\mathsf{T}A\circ A^\mathsf{T}A$, where $\circ$ is the Hadamard product.
Then,
\begin{enumerate}
\item if the $r$-th singular value of $A$ is not smaller than a constant $\kappa>0$, then $C$ is positive definite with its smallest eigenvalue not smaller than $\kappa^6$, and
\item if $r>1$ and the coherence $\mu(A)< \sqrt[3]{\frac{1}{r-1}}$, then $C$ is positive definite with its smallest eigenvalue not smaller than $1-(r-1)\mu(A)^3$.
\end{enumerate}
\end{lemma}

\begin{proof}
For the first one, since the $r$-th singular value of $A$ is greater than $\kappa>0$, we see that $A^\mathsf{T}A$ is positive definite with the smallest eigenvalue being greater than $\kappa^2$. It follows from \cite{hiai2017eigenvalue} that the smallest eigenvalue of $C$ is not smaller than the smallest eigenvalue of $(A^\mathsf{T}A)^3$, which is greater than $\kappa^6$.

The second one follows from a similar proof as that for \cite[Theorem~25]{lim2013blind}.
The conclusion follows.
\end{proof}

For $r>1$, the set of rank at most $r$ tensors has complicated geometry \cite{DL-08,L-12}. It may happen that the set of factor vectors $\{\mathbf{x}^{(k)}_1,\dots,\mathbf{x}^{(k)}_r\}$ for a best rank-$r$ approximation tensor sequence approaches to the boundary of the
set of matrices with rank at most $r$. Thus, in addition to the existence Assumption~\ref{assmp:well-defined}, we should also
make a well-conditioned assumption on a best rank-$r$ approximation for a given tensor $\A$.
\begin{assumption}\label{assmp:well-condition}
For a given tensor $\mathcal{A}$, there is a best rank-$r$ approximation $\mathcal{B}=\sum_{i=1}^r\lambda_i\mathbf{x}_i^{\otimes 3}$ such that
\begin{enumerate}
\item either the $r$-th singular value of the factor matrix
$A:=[\mathbf{x}_1,\dots,\mathbf{x}_r]$ is greater than a constant $\kappa(\mathcal{A})>0$,\label{assmp:nonsingular}
\item or the coherence $\mu(A)< \sqrt[3]{\frac{1}{r-1}}$ and $r>1$. \label{assmp:coherence}
\end{enumerate}
\end{assumption}

Assumption~\ref{assmp:well-condition} actually indicates that a rank one decomposition of a best rank-$r$ approximant of $\mathcal{A}$ is well-conditioned. It only needs the existence of such a best approximant, and do not require that all best approximants satisfy this condition. If Assumption~\ref{assmp:well-condition} is satisfied, we define
\begin{equation}\label{eq:tau-a}
\tau(\A):=\begin{cases}\max\{\kappa(\A)^6,1-(r-1)\mu(A)^3\}&\text{if both \eqref{assmp:nonsingular} and \eqref{assmp:coherence} hold},\\ \kappa(\A)^6&\text{if only \eqref{assmp:nonsingular} holds},\\1-(r-1)\mu(A)^3&\text{if only \eqref{assmp:coherence} holds}.  \end{cases}
\end{equation}

Note that Proposition~\ref{prop:app-sigma} does not require the flatness condition.
However, if $r>1$, then we need to assume the flatness condition.
\begin{proposition}\label{prop:appr-rank-r}
Let $\A\in\SS^3(\mathbb R^n)$ be nonzero and have rank greater than two, $r\geq 2$, $\overline{\mathcal{B}}$ be a best rank-$r$ approximant of $\A$ satisfying Assumption~\ref{assmp:well-condition}, and $(B,\y,X)$ be an optimal solution of \eqref{eq:non-opt-rank-reform} with $\sigma\in (0, \frac{\tau(\A)\rho(\A)}{2r})$. Suppose that $\y$ satisfies $\operatorname{rank}(\mathcal M_{k-1}(\y))=\operatorname{rank}(\mathcal M_k(\y))\leq r$. Then $\mathcal{B}$ gives a $\alpha$-quasi-optimal rank-$r$ approximation of $\mathcal A$ with $\alpha$ given by
\begin{equation*}
\alpha:=2\sqrt{\frac{r}{\tau(\A)}}\bigg(\Big(1-\sqrt{\frac{\tau(\A)}{r}}\Big)\|\A\|+2\sigma\bigg)\sigma.
\end{equation*}
\end{proposition}

\begin{proof}
By the flatness hypothesis, we know that $\mathbf{y}$ has an atomic measure with rank $s$ at most $r$.  Let it be $\mu=\sum_{i=1}^s\lambda_i\delta_{\mathbf{x}_i}$. Then $B=\sum_{i=1}^s\lambda_i\mathbf{x}_i\big(\mathbf{x}_i^{\otimes 2}\big)^\mathsf{T}$, and we have that
\begin{equation}\label{eq:quai-rank-r}
\frac{1}{2}\|M(\mathcal A)-B\|^2+\sigma\mathbf{e}^\mathsf{T}\lambda \leq \frac{1}{2}\|M(\mathcal{A})-\overline{B}\|^2+\sigma\mathbf{e}^\mathsf{T}\overline{\lambda},
\end{equation}
where $\mathbf{e}$ is the vector of all ones,
$\overline{B} = \sum_{i=1}^r \overline{\lambda}_i \overline{\x}_i (\overline{\x}_i^{\otimes 2})^\top$ (with $\overline{\lambda}_i > 0$ and $\|\overline{\x}_i \|=1$ for all $i=1,\ldots,r$) corresponds to the best rank-$r$ approximation
$\overline{\mathcal{B}}$ of $\mathcal A$ and $\overline{\lambda}$ is the corresponding positive coefficients. Then, we have
\[
  \frac{1}{2}\|M(\mathcal A)-B\|^2\leq \frac{1}{2}\|M(\mathcal{A})-\overline{B}\|^2+\sigma\big(\mathbf{e}^\mathsf{T}\overline{\lambda}-\mathbf{e}^\mathsf{T}\lambda\big).
\]
We note that possibly, $\lambda$ and $\overline{\lambda}$ have different lengths.
An observation is that $\mathbf{e}^\mathsf{T}\overline{\lambda}$ for a best rank-$r$ approximation is not smaller than the largest $\mathbf{e}^\mathsf{T}\lambda$ over all optimal solutons of \eqref{eq:non-opt-rank-reform} satisfying the flatness condition, since otherwise the optimality is violated.

By the optimality of the best rank-$r$ approximation, we have
\[
\frac{1}{2}\|M(\mathcal A)-\overline{B}\|^2=\frac{1}{2}\|M(\mathcal{A})\|^2-\frac{1}{2}\overline{\lambda}^\mathsf{T}\overline{C}\overline{\lambda},
\]
where $\overline{C}:=\overline{A}^\mathsf{T}\overline{A}\circ \overline{A}^\mathsf{T}\overline{A}\circ \overline{A}^\mathsf{T}\overline{A}$ with $\overline{A}:=[\overline{\mathbf{x}}_1,\dots,\overline{\mathbf{x}}_r]$.
Thus,
\[
\overline{\lambda}^\mathsf{T}\overline{C}\overline{\lambda}\leq \|\A\|^2.
\]
By Lemma~\ref{lem:kappa} and the hypothesis, we know that the smallest eigenvalue of the positive definite matrix $\overline{C}$ is lower bounded by $\tau(\A)$ given by \eqref{eq:tau-a}. Therefore,
\begin{equation}\label{eq:lambda-1}
\|\overline{\lambda}\|_1 \;=\; \mathbf{e}^\mathsf{T}\overline{\lambda}\leq \sqrt{r}\|\overline{\lambda}\|\leq \frac{\sqrt{r}}{\sqrt{\tau(\A)}}\|\A\|.
\end{equation}

On the other hand, we have
\begin{eqnarray*}
\rho(\A)^2 < \overline{\lambda}^\mathsf{T}\overline{C}\overline{\lambda}
 \leq \sum_{i=1}^r\sum_{j=1}^r \overline{\lambda}_i \overline{\lambda}_j |\overline{C}_{ij}|
 \leq
  \sum_{i=1}^r\sum_{j=1}^r \overline{\lambda}_i \overline{\lambda}_j  =
  \|\overline{\lambda}\|_1^2,
\end{eqnarray*}
where the strict inequality follows from the fact that a best rank $r\geq 2$ approximation must be strictly better than the best rank one approximation, and the second inequality from $|\overline{C}_{ij}| = |(\overline{A}^\top\overline{A})_{ij}|^3
= |\overline{\x}_i^\top \overline{\x}_j|^3 \leq 1$. Thus, $\|\overline{\lambda}\|_1 \geq \rho(\A)$.
Similarly, we also have $\lambda^\top C \lambda \leq \|\lambda \|_1^2$,
where $C:=A^\mathsf{T}A\circ A^\mathsf{T}A\circ A^\mathsf{T}A$ with $A:=[\mathbf{x}_1,\dots,\mathbf{x}_s]$.

Expanding the left hand side of the inequality \eqref{eq:quai-rank-r}, it becomes
\begin{equation}\label{eq:lambda}
\frac{1}{2}\|M(\mathcal{A})\|^2-\sum_{i=1}^s\lambda_i\langle\mathcal{A},\mathbf{x}_i^{\otimes 3}\rangle+\frac{1}{2}\lambda^\mathsf{T}C\lambda+\sigma\mathbf{e}^\mathsf{T}\lambda.
\end{equation}

Note that each $\lambda_i>0$, and thus by the optimality of $\lambda$ (i.e., setting the derivative of the above expression with respect to $\lambda$ to zero), we must have
\begin{equation*}
  \sum_{i=1}^s\lambda_i\langle\mathcal{A},\mathbf{x}_i^{\otimes 3}\rangle-\sigma\mathbf{e}^\mathsf{T}\lambda=\lambda^\mathsf{T}C\lambda.
\end{equation*}
Consequently, we have
\[
  \frac{1}{2}\|M(\mathcal A)-B\|^2+\sigma\mathbf{e}^\mathsf{T}\lambda=\frac{1}{2}\|M(\mathcal{A})\|^2-\frac{1}{2}\lambda^\mathsf{T}C\lambda.
\]
Thus, from \eqref{eq:quai-rank-r}, we have
\[
-\frac{1}{2}\lambda^\mathsf{T}C\lambda\leq -\frac{1}{2}\overline{\lambda}^\mathsf{T}\overline{C}\overline{\lambda}+\sigma\mathbf{e}^\mathsf{T}\overline{\lambda}.
\]
This, together with the fact that $\lambda^\top C \lambda \leq \|\lambda \|_1^2$, implies
\begin{align}\label{eq:lambda-3}
\|\lambda\|_1^2&\geq  \lambda^\mathsf{T}C\lambda\geq \overline{\lambda}^\mathsf{T}\overline{C}\overline{\lambda}-2\sigma\mathbf{e}^\mathsf{T}\overline{\lambda}\nonumber\\
&\geq \tau(\A)\|\overline{\lambda}\|^2-2\sigma\mathbf{e}^\mathsf{T}\overline{\lambda}\geq \frac{\tau(\A)}{r}\|\overline{\lambda}\|_1^2-2\sigma\|\overline{\lambda}\|_1\nonumber\\
&=\frac{\tau(\A)}{r}\Big(\|\overline{\lambda}\|_1^2-4\frac{r\sigma}{\tau(\A)}\|\overline{\lambda}\|_1+\big(\frac{2r\sigma}{\tau(\A)}\big)^2
+2\frac{r\sigma}{\tau(\A)}\|\overline{\lambda}\|_1-\big(2\frac{r\sigma}{\tau(\A)}\big)^2\Big).
\end{align}
This, together with \eqref{eq:lambda-1} and $\|\overline{\lambda}\|_1\geq\rho(\A)$, implies that when $\sigma\leq \frac{\tau(\A)\rho(\A)}{2r}$, we have
\begin{equation}\label{eq:lambda-2}
\|\lambda\|_1\geq \sqrt{\frac{\tau(\A)}{r}}\Big(\|\overline{\lambda}\|_1-2\frac{r}{\tau(\A)}\sigma\Big).
\end{equation}

It follows from \eqref{eq:quai-rank-r}, \eqref{eq:lambda-1}, and \eqref{eq:lambda-2} that
\begin{equation}\label{eq:app-rank-r}
\|M(\mathcal A)-B\|^2\leq \|M(\mathcal{A})-\overline{B}\|^2+2\sigma\bigg(\Big(1-\sqrt{\frac{\tau(\A)}{r}}\Big)\sqrt{\frac{r}{\tau(\A)}}\|\A\|+2\sqrt{\frac{r}{\tau(\A)}}\sigma\bigg).
\end{equation}
This completes the proof.
\end{proof}

We can apply a similar refinement technique as Proposition~\ref{prop:app-sigma} to improve the quality of $B$ given in Proposition~\ref{prop:appr-rank-r}. Actually, by the optimality of $\lambda$ in \eqref{eq:lambda}, we have that $\lambda=C^{-1}(\mathbf{u}-\sigma\mathbf{e})$ and hence
\[
\|M(\mathcal A)-B\|^2=\|\A\|^2-2\lambda^\mathsf{T}\mathbf{u}+\lambda^\mathsf{T}C\lambda=\|\A\|^2-\mathbf{u}^\mathsf{T}C^{-1}\mathbf{u}+\sigma^2\mathbf e^\mathsf{T}C^{-1}\mathbf e,
\]
where $\mathbf u:=(\A\x_1^3,\dots,\A\x_s^3)^\mathsf{T}$. Given the vectors $\x_1,\dots,\x_s$, we can optimize their coefficients to an approximant $\mathcal B':=\sum_{i=1}^s\mu_i\x_i^{\otimes 3}$ such that
\[
\|\A-\mathcal B'\|^2=\|\A\|^2-\mathbf{u}^\mathsf{T}C^{-1}\mathbf{u}.
\]
This can reduce at least the amount of
 $\sigma^2$ from \eqref{eq:app-rank-r} since $\mathbf e^\mathsf{T}C^{-1}\mathbf e\geq 1$.
Note that when $r=1$, we can take $\tau(\A)=\kappa(\A)=1$, and we thus get a $4\sigma^2$-quasi-optimality estimation from \eqref{eq:app-rank-r}. While, even with the above refinement, we can only get $3\sigma^2$-quasi-optimality; but we know a best rank one approximant can be recovered by Proposition~\ref{prop:app-sigma}. This follows largely by the estimations from \eqref{eq:lambda-3} and \eqref{eq:lambda-2}, and it indicates that there are some room for improvement.
Nevertheless, the next result shows that the estimation of $\sigma^2$ in Proposition~\ref{prop:appr-rank-r} cannot be eliminated.
\begin{proposition}\label{prop:general-sigma}
If $\A$ is an orthogonally decomposable tensor with rank $s\leq n$, then for $\sigma\in[0,\lambda_r-\lambda_{r+1}]$\footnote{We let $\lambda_{s+1}=0$ if needed.} with $r\leq s$, $\mu:=\sum_{i=1}^r(\lambda_i-\sigma)\delta_{\x_i}$ gives a global optimizer of problem \eqref{eq:non-opt-rank-reform} and a $r\sigma^2$-quasi-optimal rank-$r$ approximation of $\A$.
\end{proposition}

\begin{proof}
Let $\A=\sum_{i=1}^s\lambda_i\x_i^{\otimes 3}$ be an orthogonal decomposition.
Let
\[
u(\x):=\sum_{i=1}^r\sigma(\x_i^\mathsf{T}\x)^3.
\]
Let $\mu:=\sum_{i=1}^r(\lambda_i-\sigma)\delta_{\x_i}$ be the $r$-atomic measure and $\bar\y$ be the moment sequence defined by the measure $\mu$. Let $U$ be the corresponding matrix for the cubic polynomial $u(\x)$.
Since $\lambda_r-\sigma\geq\lambda_{r+1}$, we have that
$\Pp_k(\bar\y)$ gives a best rank-$r$ approximation of the matrix $M(\A)-U$.

Let
\[
w(\x):=\frac{3\sigma}{2}\x^\mathsf{T}\x.
\]
Let $\bar Z$ be the moment matrix defined by the following polynomial:
\begin{equation*}
z(\mathbf x):=\sigma-u(\mathbf x)+w(\mathbf x)(\|\mathbf x\|^2-1)
\end{equation*}
Obviously,
\[
z(\x_i)=0\ \text{for all }i=1,\dots,r,
\]
so the complementarity between $\bar Z$ and $\mathcal M_k(\bar \y)$ is fulfilled.
In the following, by Theorem~\ref{thm:optimality}, we only need to check that the polynomial $w(\x)$ satisfies the fact that the polynomial $z(\x)$
is a sum of squares of polynomials.

Applying an orthogonal transformation if necessary, we can assume without loss of generality that $\x_i=\mathbf e_i$ (the $i$-th column vector of the identity matrix) for all $i=1,\dots,r$. Let the resulting polynomial be $\hat z(\x)$.
We then have
\begin{align*}
\frac{\hat z(\mathbf x)}{\sigma}=1-\sum_{i=1}^rx_i^3+\frac{3}{2}(\x^\mathsf{T}\x)(\x^\mathsf{T}\x-1).
\end{align*}
Thus, it follows from \cite[Section~7.3 (in Supplementary)]{TS-15} that the polynomial $\frac{\hat z(\mathbf x)}{\sigma}$ and hence $z(\mathbf x)$ is a sum of squares.

By Theorem~\ref{thm:optimality}, $\bar\y$ is a global minimizer of \eqref{eq:non-opt-rank-reform}. The approximation error to the best rank-$r$ approximation $\y^*$ generated by $\sum_{i=1}^r\lambda_i\delta_{\x_i}$ is given by
\[
\|\A-\mathcal B\|^2=\|\A-\mathcal B^*\|^2+r\sigma^2.
\]
The conclusion then follows.
\end{proof}

Of course, a refinement as in the preceding analysis will give the global optimal solution in the scenario of Proposition~\ref{prop:general-sigma}.
But a generic tensor does not necessarily have an orthogonal decomposition \cite{L-12}, which implies that the linear term over $\sigma$ in \eqref{eq:app-rank-r} is probably essential.


\subsection{Optimality}\label{sec:duality}
In this section, we summarize the established results into certifications on best approximations and quasi-optimal approximations of a given tensor.
\begin{theorem}[Rank-$r$ Approximation]\label{thm:best}
Suppose that $\sigma\geq 0$ and $k=2$ are chosen in \eqref{eq:non-opt-rank-reform}.
If there exists a triplet $(\bar U,\bar W, \bar Z)$ such that the feasibility \eqref{eq:lagrange-dual-compact} is satisfied and a vector $\bar \y$ such that all  the optimality condition \eqref{eq:optimality}, $\operatorname{rank}(\mathcal M_{1}(\bar\y))=\operatorname{rank}(\mathcal M_2(\bar\y))\neq r+1$, and
\begin{equation*}
\operatorname{rank}(\mathcal P_k(\bar \y))\leq r
\end{equation*}
are satisfied,
then $\bar \y$ gives an optimal solution for \eqref{eq:non-opt-rank-reform} and
\begin{enumerate}
\item if $\sigma=0$, or $r=1$ and $\sigma<\rho(\A)$, then $\mathcal P_k(\bar \y)$ gives a best rank-$r$ approximation of $\A$;
\item if $\sigma\in (0, \frac{\tau(\A)\rho(\A)}{2r})$, $\operatorname{rank}(\mathcal M_{1}(\bar\y))\leq r$, and Assumption~\ref{assmp:well-condition} is satisfied, then $\mathcal P_k(\bar \y)$ gives a $2\sqrt{\frac{r}{\tau(\A)}}\bigg(\Big(1-\sqrt{\frac{\tau(\A)}{r}}\Big)\|\A\|+2\sigma\bigg)\sigma$-quasi-optimal rank-$r$ approximation of $\mathcal A$.
\end{enumerate}
\end{theorem}

\begin{proof}
It follows from Theorem~\ref{thm:optimality}, Theorem~\ref{thm:sub-optimal}, and Proposition~\ref{prop:appr-rank-r}.
\end{proof}

The case of the best rank one approximation is more clear.
\begin{theorem}[Rank One Approximation]\label{thm:best-one}
Suppose that $\sigma<\rho(\A)$ is chosen in \eqref{eq:non-opt-rank-reform} and $r=1$.
If there exists a triplet $(\bar U,\bar W, \bar Z)$ such that the feasibility \eqref{eq:lagrange-dual-compact} is satisfied and a vector $\bar \y$ such that all the optimality condition \eqref{eq:optimality} are satisfied, and
$\operatorname{rank}(\mathcal P_k(\bar \y))\leq 1$,
then $\bar \y$ gives a best rank one approximation of \eqref{eq:best-r}.
\end{theorem}

\begin{proof}
This follows from the fact that \eqref{eq:non-opt-rank-reform} is a relaxation of \eqref{eq:best-r} and $\B$ is a tensor of rank at most one if and only if the corresponding matrix $B$ has rank at most one \cite{L-12}. Thus, the flatness condition in Theorem~\ref{thm:best} is not needed in this case. The result follows from Proposition~\ref{prop:app-sigma}.
\end{proof}

Note that for the rank one case, if $\sigma>0$ in Theorems~\ref{thm:best} and \ref{thm:best-one}, a simple refinement as in Proposition~\ref{prop:app-sigma} to get a best approximant is necessary.

Actually, the vector $\bar\y$ in both Theorem~\ref{thm:best} and \ref{thm:best-one} need not be a solution to problem \eqref{eq:non-opt-rank-reform}. We can have such a moment sequence by other means or methods, and we can also check the optimality for it using these theorems. This is exactly Theorem~\ref{thm:best-intro} when a candidate tensor is available.

\section{Numerical Illustration}\label{sec:numerical}
In this section, we present some numerical examples to illustrate the 
usefulness
of the theoretical results presented so far.

The emphasis is put on certifying the global optimality for the best low rank tensor computed, which is achieved by solving the primal problem \eqref{eq:non-opt-rank-reform} and employing the dual problem \eqref{eq:lagrange-dual-min} to check this optimality.
Note that both problems \eqref{eq:non-opt-rank-reform} and \eqref{eq:lagrange-dual-min} are not easy to solve, the primal has a rank constraint and the dual has a nonsmooth objective function. The design of an efficient numerical algorithm for solving
both problems will be addressed in another paper.
We will apply existing methods for \eqref{eq:non-opt-rank-reform} in the current paper. For the sake of not lengthening the paper, we do not include the full details, but just give
a brief description of the implementation here.

The dual problem \eqref{eq:lagrange-dual-min} will not be solved in this paper. Thus, it may happen that the global optimal solution is found but it cannot be certified, e.g., see Table~\ref{table:exm-2}. Instead, we will apply a DCA framework together with sGS-ADMM
to solve a penalized version of the
primal problem \eqref{eq:non-opt-rank-reform} and use the generated dual multipliers to check the optimality of the dual problem \eqref{eq:lagrange-dual-min}.

\subsection{Algorithmic Rationale}\label{sec:algorithm}
In the following, for simplicity, we omit the subscript $k$ for the order of relaxation
in  \eqref{eq:non-opt-rank-reform}. In this section, $k=2$ is applied.
For problem \eqref{eq:non-opt-rank-reform}, we first rewrite it as
\begin{equation*}
\begin{array}{rl}
\min& \frac{1}{2}\|M(\A)-B\|^2+\delta_{\SS^n_+}(X)+\sigma\langle E_0,X\rangle
\\[3pt]
\text{s.t.}
&B-\mathcal P(\y)=0,\  X-\mathcal M(\y)=0,\  \mathcal L(\y)=0,\\
&\|B\|_*-\|B\|_{(r)}=0,
\end{array}
\end{equation*}
where $\|B\|_*$ is the \textit{nuclear norm} of $B$, i.e., the sum of all the singular values of $B$, and $\|B\|_{(r)}$ is the sum of the $r$ largest singular values of $B$, known as the \textit{Ky-Fan $r$-norm} \cite{HJ-85}.
Based on the above formulation, we apply a penalty approach by solving
\begin{equation}\label{eq:non-opt-rank-nu-pen}
\begin{array}{rl}
\min& \frac{1}{2}\|M(\A)-B\|^2+\delta_{\SS^n_+}(X)+\sigma\langle E_0,X\rangle+\rho(\|B\|_*-\|B\|_{(r)})\\[3pt]
\text{s.t.}
&B-\mathcal P(\y)=0,\  X-\mathcal M(\y)=0,\  \mathcal L(\y)=0,
\end{array}
\end{equation}
where $\rho>0$ is a penalty parameter for the constraint $\|B\|_*-\|B\|_{(r)}=0$. Note that $\|B\|_*-\|B\|_{(r)}\geq 0$ holds always.
We see that the objective function of \eqref{eq:non-opt-rank-nu-pen} is a \textit{difference of convex} (DC) function.
Then, a standard DC algorithm (DCA) is applied for problem \eqref{eq:non-opt-rank-nu-pen}. 
The important issue on the exactness of the penalty for \eqref{eq:non-opt-rank-nu-pen} will be addressed in another paper too. From the numerical experiments, it seems that
the exact penalty property holds,
at least for some classes of given data.

The DCA framework is as follows: at each iteration point $(B_k,X_k,\y_k)$, we linearize $\|B\|_{(r)}$ as
\[
\|B_k\|_{(r)}+\langle C_k,B\rangle
\]
for some $C_k\in\partial \|B_k\|_{(r)}$. Problem \eqref{eq:non-opt-rank-nu-pen} is then replaced by a DC linearization
\begin{equation}\label{eq:non-opt-rank-nu-linear}
\begin{array}{rl}
\min& \frac{1}{2}\|M(\A)-B\|^2+\sigma\langle E_0,X\rangle+\delta_{\SS^n_+}(X)+\rho(\|B\|_*-\langle C_k,B\rangle)\\[5pt]
\text{s.t.}
&B-\mathcal P(\y)=0,\  X-\mathcal M(\y)=0,\ \mathcal L(\y)=0.
\end{array}
\end{equation}
This procedure is a standard algorithm for DC programs, whose convergence theory is classical \cite{PL-97}.

Problem \eqref{eq:non-opt-rank-nu-linear} is a linearly constrained convex matrix optimization problem with a nonsmooth objective function. The variables can be grouped into two sets $\{B,\y\}$ and $\{X\}$. Corresponding to each set, the objective function has a nonsmooth part.
We apply the symmetric Gauss-Siedel alternating direction method of multipliers (sGS-ADMM) \cite{LST-16} to solve \eqref{eq:non-opt-rank-nu-linear} based on the grouping
of the above two sets of variables.

The problem \eqref{eq:non-opt-rank-nu-linear} is solved with the following optimality being satisfied
\begin{equation}\label{eq:primal-optimal}
\begin{array}{c}
0\preceq \sigma E_0-V\perp X\succeq 0,\\
C_k+\frac{1}{\rho}(M(\mathcal A)-U-B)\in\partial \|B\|_*,\\
\mathcal L^*(W)+\mathcal P^*(U)-\mathcal M^*(V)=0,\\
B-\mathcal P(\y)=0,\ X-\mathcal M(\y)=0,\ \text{and }\mathcal L(\y)=0
\end{array}
\end{equation}
for some multipliers $(U,V,W)$. Note that if \eqref{eq:primal-optimal} is satisfied, then for $(U,W,Z)$ with $Z:=\sigma E_0-V$, the feasibility for the dual problem \eqref{eq:lagrange-dual-min} is fulfilled.

We briefly outline the relationship between the target dual \eqref{eq:lagrange-dual-min} and the dual of \eqref{eq:non-opt-rank-nu-linear}, which is
\begin{equation}\label{eq:non-opt-rank-lin-dual}
\begin{array}{rl}
\max& -\frac{1}{2}\|\operatorname{T}_{\frac{1}{\rho}}(M(\mathcal A)-U+\rho C_k)\|^2
\\[3pt]
\text{s.t.}
&\mathcal M^*(Z)+\mathcal P^*(U)-\mathcal L^*(W)=\sigma\mathcal M^*(E_0),\\
& Z\succeq 0,
\end{array}
\end{equation}
where $\operatorname{T}_{\frac{1}{\rho}}$ is the matrix \textit{soft-threshold} operator defined as
\begin{equation*}
\operatorname{T}_{\tau}(A):=U\operatorname{diag}(\max(\mathbf d-{\tau}^{-1}\mathbf e,\mathbf 0))V^\mathsf{T},
\end{equation*}
where $U\operatorname{diag}(\mathbf d)V^\mathsf{T}=A$ is a singular value decomposition of $A$ with $\mathbf d$ being the vector of singular values.
We see that \eqref{eq:non-opt-rank-lin-dual} is an approximation to \eqref{eq:lagrange-dual-min}, and they are equivalent if $C_k$ is well-chosen.

\subsection{Illustrative Examples}\label{sec:examples}
All the tests were conducted on a Lenovo laptop with 128GB RAM  and 2.8GHz E-2276M CPU running 64bit Windows operation system. All codes were written in {\sc Matlab}. The default parameters are chosen as $\tau=1.5,\kappa=1,\gamma=0.001,\sigma=0.00001,\rho=1$, where $\tau$ is the steplength in sGS-ADMM, $\gamma$ is a proximal parameter, and $\kappa$ is the penalty parameter for the augmented Lagrangian function of \eqref{eq:non-opt-rank-nu-linear}.
In the examples, the \textit{duality gap} refers to the difference $\psi(B,X)-\phi(U)$ as defined respectively in
\eqref{eq:non-opt-rank-reform}
and
\eqref{eq:lagrange-dual-compact};
the \textit{feasibility} refers to the maximum of the primal feasibility and the dual feasibility violations; the \textit{psd residaul} refers to the violation of the first condition in \eqref{eq:primal-optimal}; and the \textit{rank residual} refers to the violation of the last condition in \eqref{eq:optimality}.

We see that when all the duality gap,  feasibility violation, psd residual and rank residual are small, and the rank of $B$ is bounded by $r$, then both the primal problem \eqref{eq:non-opt-rank-reform} and the dual problem \eqref{eq:lagrange-dual-min} are solved globally. If furthermore the flatness condition is satisfied or $\rk(B)=1$, then the original best rank-$r$ approximation problem is solved globally with good quality (cf.\ Theorem~\ref{thm:best} and Theorem~\ref{thm:best-one}).

\begin{example}\label{exm:o-1}
{\em
This example is taken from De Lathauwer, De Moor and Vandewalle \cite[Example~5]{De-De-Van:bes}. It is a tensor in $\SS^3(\mathbb R^2)$ with the independent elements being
\[
a_{111}=2,\ a_{112}=1,\ a_{122}=1,\ \text{and }a_{222}=1.
\]
The best rank one approximation computed is
\[
\lambda=3.2560\ \text{with }\x= (0.7981,    0.6025)^\mathsf{T},
\]
which is exactly the one given in \cite{De-De-Van:bes}. The global optimality is
certified with the duality gap $= 1.6\times 10^{-8}$, feasibility $=5.7\times 10^{-8}$, psd residual $= 1.6\times 10^{-8}$, rank residual $= 5.3\times 10^{-10}$, and the computed matrix $B$ having rank one. The approximation residual is $\|\mathcal A-\mathcal B\|=0.555$.
}
\end{example}

\begin{example}\label{exm:o-5}
{\em
This example tests a set of perturbed versions of Example~\ref{exm:o-1}. They are tensors in $\SS^3(\mathbb R^2)$ with the independent elements being
\[
a_{111}=2,\ a_{112}=1,\ a_{122}=1-\epsilon,\ \text{and }a_{222}=1+\epsilon,
\]
where $\epsilon>0$ is a perturbation in $[10^{-6},10^{-1}]$. We tested $100$ instances, each taking an $\epsilon\in [10^{-6},10^{-1}]$, starting from $10^{-6}$ with an equal difference $10^{-3}$. In each case, the method successfully computed the best rank one approximation, together with a global optimality certification as in Example~\ref{exm:o-1}. We do not present the similar but tedious data, while show the computed $\lambda$, and the coordinates of the vector $\x$ in Figure~\ref{fig:perturbation}, from which we can see the evolution of the optimal solutions along the perturbations.
}
\begin{figure*}
	\centering
\subfigure[The norm $\lambda$]{ \includegraphics[width=0.3\columnwidth]{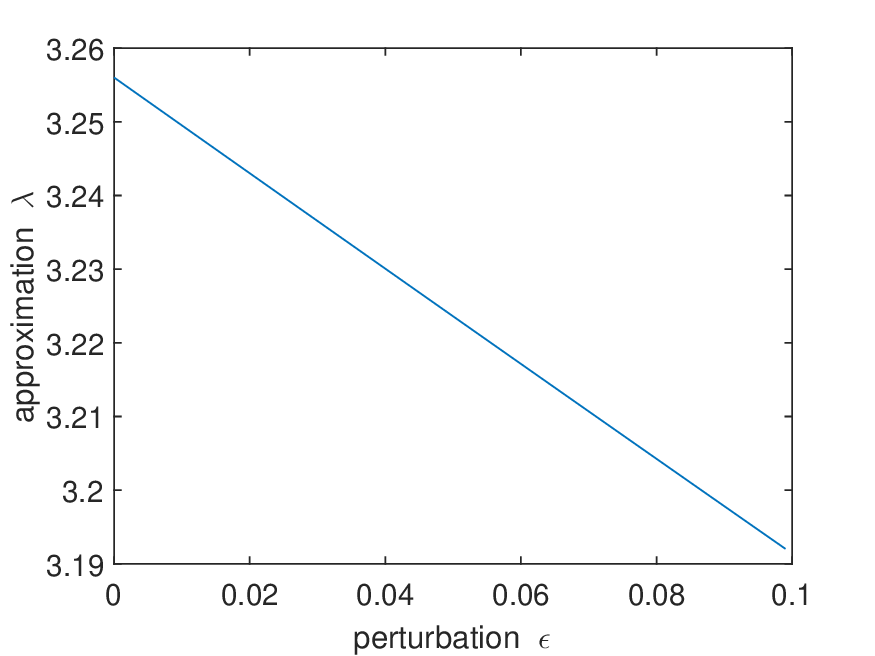}
}
\subfigure[The first coordinate $x_1$]{  \includegraphics[width=0.3\columnwidth]{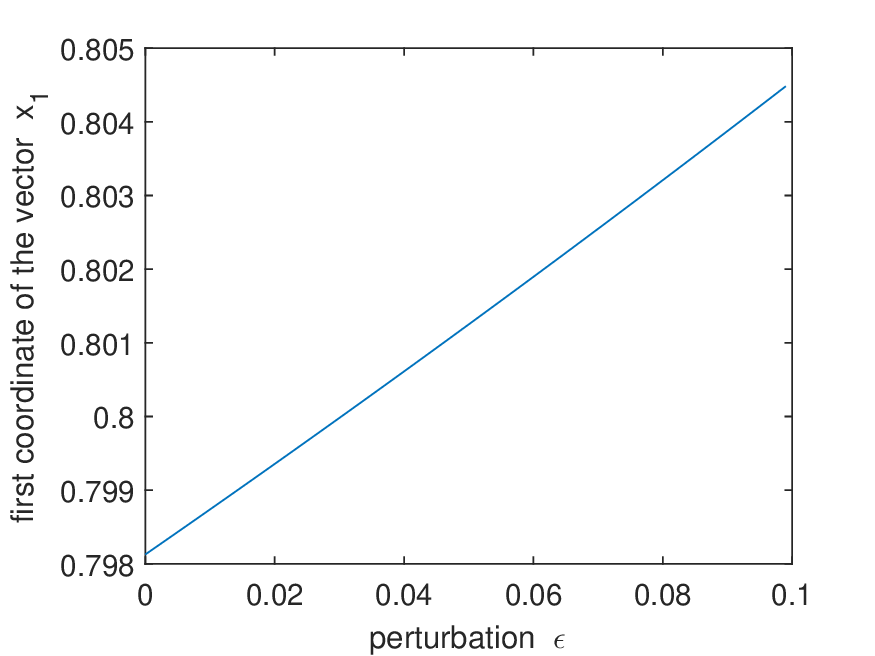}
}
\subfigure[The second coordinate $x_2$]{  \includegraphics[width=0.3\columnwidth]{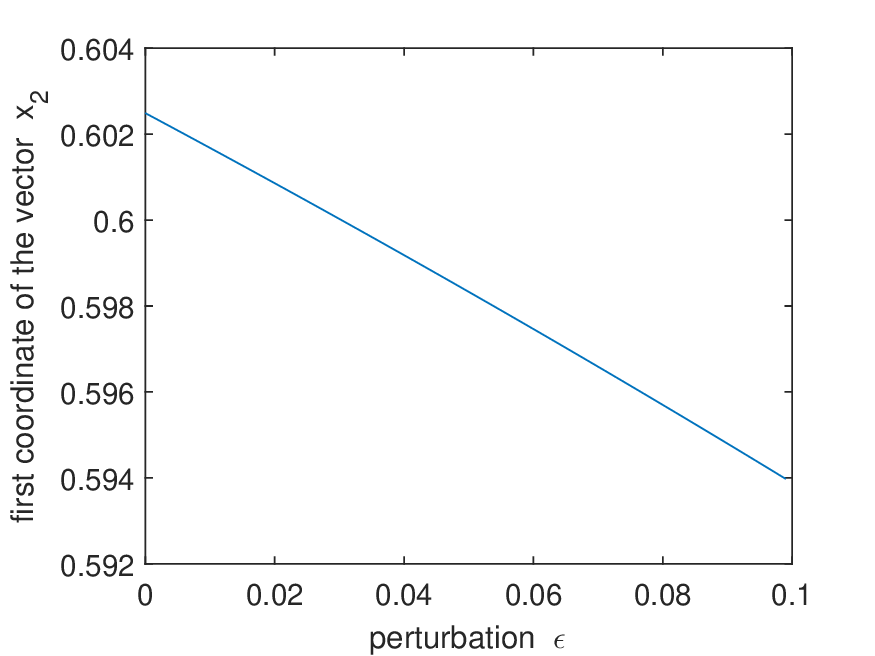}
}
	\caption{The computed best rank one tensors along the perturbations.}
	\label{fig:perturbation}
\end{figure*}
\end{example}
\begin{example}\label{exm:o-2}
{\em
This tensor is taken from Qi \cite[Example~2]{Qi:rat} as well as Nie and Wang \cite[Example~3.3]{NW-14}.
This is a tensor in $\SS^3(\mathbb R^3)$ with the independent elements being
\begin{align*}
&a_{111}=0.0517,\ a_{112}=0.3579,\ a_{113}=0.5298,\ a_{122}=0.7544,\ a_{123}=0.2156,
\\
&a_{133}=0.3612,\ a_{222}=0.3943,\ a_{223}=0.0146,\ a_{233}=0.6718,\ a_{333}=0.9723.
\end{align*}
The best rank one approximation computed is
\[
\lambda=2.1110\ \text{with }\x= (0.5204,    0.5113,    0.6839)^\mathsf{T},
\]
which is certified with the duality gap $=6.1\times 10^{-8}$, feasibility $=3.4\times 10^{-8}$, psd residual $=6.1\times 10^{-8}$, rank residual $=3.5\times 10^{-10}$,
 and the computed matrix $B$ having rank one.
The result agrees with that in \cite[Example~3.3]{NW-14}. The approximation residual is $1.0199$.
}
\end{example}

\begin{example}\label{exm:o-3}
{\em
This is a tensor in $\SS^3(\mathbb R^2)$ with the independent elements being
\[
a_{111}=0.1615 ,\ a_{112}=0.5603 ,\ a_{122}=-0.5824,\ \text{and }a_{222}= 1.2076.
\]
This tensor has rank three and has a rank decomposition given by
\[
0.6648\times \begin{bmatrix} -0.3314\\  0.5551\end{bmatrix}^{\otimes 3}+
    0.1132\times \begin{bmatrix} 0.6042\\  0.0891\end{bmatrix}^{\otimes 3}+
    0.0828\times \begin{bmatrix} 0.57064\\  0.4388\end{bmatrix}^{\otimes 3}.
\]
The computed best rank two approximation tensor $\mathcal B$ has independent elements
\[
b_{111}= 0.4771,\ b_{112}=0.5603,\ b_{122}=-0.4771,\ \text{and }b_{222}= 1.2076.
\]
It is a rank two tensor, and has a rank decomposition given as
\[
 1.6347\times \begin{bmatrix} -0.4514\\ 0.8923\end{bmatrix}^{\otimes 3}+
    0.8000\times \begin{bmatrix} 0.9222\\  0.3867\end{bmatrix}^{\otimes 3}.
\]
The approximation residual is $0.3327$. The duality gap is $4.2\times 10^{-8}$ with the dual objective function value $=0.0664$.
The feasibility $=2.5\times 10^{-8}$, psd residual $=4.4\times 10^{-8}$, and rank residual  $=2.7\times 10^{-8}$. The computed moment vector $\y$ is
\begin{align*}
\small
(&2.4345,   -0.0000,    1.7678,    1.0134,   -0.3730,    1.4211,    0.4771,
   0.5603,\\
   &  -0.4771,    1.2076,    0.6465,    0.1085,    0.3669,   -0.4815,    1.0542)^\mathsf{T},
\end{align*}
from which we can see that the (numerically) nonzero eigenvalues of the first moment matrix $\mathcal M_1(\y)$ are $1.0838,  3.7852$, and those of the second moment matrix $\mathcal M_2(\y)$ are $1.7691,     5.1674$. Therefore, the flatness condition is satisfied and hence the quantified optimality is certified.
}
\end{example}

\begin{example}\label{exm:o-4}
{\em
This is a tensor in $\SS^3(\mathbb R^3)$ with the independent elements being
\begin{align*}
&a_{111}=-0.5570,\ a_{112} =1.2276,\ a_{113}= 0.0843,\ a_{122}=    -0.1576,\ a_{123}=   -0.1005,\\
& a_{133}=  0.1446,\ a_{222}=    1.6399,\ a_{223}=     0.2158,\ a_{233}=    -0.3820,\ a_{333}=    1.0080.
\end{align*}
This tensor has rank three and has a rank decomposition given by
\[
0.3474\times \begin{bmatrix} 0.0973\\   -0.2534 \\   0.6619\end{bmatrix}^{\otimes 3}+
    0.6606\times \begin{bmatrix} 0.3489 \\   0.5181 \\   0.0087\end{bmatrix}^{\otimes 3}+
    0.3839\times \begin{bmatrix}  -0.6028 \\   0.5873 \\   0.0398\end{bmatrix}^{\otimes 3}.
\]
The computed best rank-two approximation tensor $\mathcal B$ has independent elements
\begin{align*}
&a_{111}=-0.0085,\ a_{112} =1.2274,\ a_{113}= 0.0965,\ a_{122}=    0.0178,\ a_{123}=   -0.0584,\\
&  a_{133}=  -0.0093,\ a_{222}=    1.6536,\ a_{223}=     0.1286,\ a_{233}=     0.0129,\ a_{333}=    0.0015.
\end{align*}
It is a rank two tensor, and has a rank decomposition given by
\[
 1.9223\times \begin{bmatrix} 0.6486 \\   0.7606 \\   0.0279\end{bmatrix}^{\otimes 3}+
    1.9063\times \begin{bmatrix} -0.6539  \\  0.7511 \\   0.0907\end{bmatrix}^{\otimes 3}.
\]
The approximation residual is $1.2560$. The duality gap $= 1.6\times 10^{-7}$ with the dual objective function value $=0.9895$.
The feasibility $= 3.6\times 10^{-8}$, psd residual $= 1.6\times 10^{-7}$, rank residual $= 9.5\times 10^{-11}$. The computed moment vector $\y$ is
\begin{align*}
\small
(&3.8286,    0.0000,    2.8939,    0.2265,    1.6239,    0.0118,   -0.0782,    2.1875,    0.1707,\\
  &  0.0172,   -0.0085,    1.2274,    0.0965,    0.0178,   -0.0584,   -0.0093,    1.6536,    0.1286,\\
   & 0.0129,    0.0015,    0.6888,   -0.0015,   -0.0337,    0.9277,    0.0727,    0.0073,    0.0202,\\
   &-0.0437,   -0.0070,   -0.0009,    1.2501,    0.0969,    0.0097,    0.0011,    0.0001)^\mathsf{T},
\end{align*}
from which we can see that the (numerically) nonzero eigenvalues of the first moment matrix $\mathcal M_1(\y)$ are $2.5635,    7.9775$, and those of the second moment matrix $\mathcal M_2(\y)$ are $1.6277, 6.0295$. Therefore, the flatness condition is satisfied and hence the quantified optimality is certified.
}
\end{example}

\begin{example}\label{exm:1}
{\em
This is a tensor $\A\in\SS^3(\mathbb R^2)$ with the independent elements being
\[
a_{111}=0.5662,\ a_{112} = -0.0971,\ a_{122} = 0.0713,\ \text{and }a_{222} = 0.2664.
\]
This tensor is an orthogonally decomposable tensor with rank two.
The computed best rank two approximation is given by
\begin{equation*}
\B=0.5945\times \begin{bmatrix} 0.9827\\ -0.1854\end{bmatrix}^{\otimes 3}+0.2849\times \begin{bmatrix} 0.1854\\ 0.9827\end{bmatrix}^{\otimes 3}
\end{equation*}
with the approximation residual being $\|\A-\B\|=1.1\times 10^{-4}$.
We see that $\B$ is an orthogonally decomposable tensor. The duality gap $=9.7\times 10^{-9}$, feasibility $=7.4\times 10^{-11}$, psd residual $=1.2\times 10^{-10}$, and rank residual $=1.3\times 10^{-4}$. The rank residual indicates that the DCA did not return an accurate solution in this case. But, the approximation quality is consistent with the theoretical bound given in Proposition~\ref{prop:general-sigma}.

We also tested its best rank one approximation. Ten simulations were drawn, nine of them find the global solution, and the rest finds a local minimizer.

The best rank one approximation computed is
\[
\B=0.5949\times\x^{\otimes 3}\ \text{with }\x= (0.9826, -0.1859)^\mathsf{T}
\]
with the approximation residual being $\|\A-\B\|=0.2848$.  The duality gap $=1.5\times 10^{-8}$, feasibility $=1.4\times 10^{-8}$, psd residual  $=1.5\times 10^{-8}$, and  rank residual  $=2.5\times 10^{-12}$. We see that the approximation quality is very good.

Since the given tensor is orthogonally decomposable, we know all the local minimizers \cite{HL-18}.
A local minimizer is
\[
\B=0.2848\times \x^{\otimes 3}\ \text{with }\x=(0.1861, 0.9825)^\mathsf{T}
\]
with the approximation residual being $\|\A-\B\|=0.5949$.  The duality gap $=0.1429$,  feasibility $=5.2\times 10^{-9}$,  psd residual $=2.0\times 10^{-9}$, rank residual $=0.3209$. Thus, it cannot be certified as a global optimizer by the theory established in this paper, which agrees with the observed fact.
}
\end{example}

\begin{example}\label{exm:2}
{\em
The tensors in this example are perturbed variations of the tensor in Example~\ref{exm:1}.
We test six variations, by adding to each component of the tensor in Example~\ref{exm:1} with $\epsilon=10^{-1},10^{-2},10^{-3},10^{-4},10^{-5}$ and $10^{-6}$ respectively.
For each case, we run  the algorithm $100$ times with random initialization. We record the norm of the best rank one approximation tensor found by the algorithm.
\begin{table}[h!]\caption{Performance of the perturbed tensor.}\label{table:exm-2}
\centering
\begin{tabular}{|c | c | c | c | c | c | c |}
\hline
$\epsilon$& \multicolumn{2}{c|}{$10^{-1}$}  & \multicolumn{2}{c|}{$10^{-2}$} &\multicolumn{2}{c|}{$10^{-3}$} \\\hline
$\|B\|$& 0.66604&0.66604&0.59994 & 0.30133  &0.59525 &0.28633 \\ \hline
\text{Num}& 82&18& 84&16 &77 &23  \\ \hline
\text{gap}& $1.6\times 10^{-8}$&0.032& $1.4\times 10^{-8}$&0.159&$1.5\times 10^{-8}$&0.138 \\ \hline \hline
$\epsilon$&\multicolumn{2}{c|}{$10^{-4}$} &\multicolumn{2}{c|}{$10^{-5}$}&\multicolumn{2}{c|}{$10^{-6}$}\\ \hline
$\|B\|$&0.59479 &0.28489 &0.59474 &0.28474 &0.59474 &0.28473\\\hline
\text{Num}&80 &20 &76 &24 &70 &30  \\ \hline
\text{gap}&$1.5\times 10^{-8}$&0.165&$1.6\times 10^{-8}$& 0.151&$1.5\times 10^{-8}$&0.149\\ \hline
\end{tabular}
\end{table}In all the cases, \textit{Num} represents the number of times that the algorithm finds the rank one tensor with the outlined norm \textit{$\|B\|$}, which is the global solution when we group it into the \textit{gap} of magnitude $10^{-8}$.
Note that when the perturbation is $0.1$, the resulting tensor can be considered as ``generic". It is seen that
in this case, global optimizers are found in all the simulations, but our algorithm cannot certify this fact in $18$ simulations. We think this is due to the fact that the dual problem \eqref{eq:lagrange-dual-min} is not solved with a guarantee in these cases. However, for perturbations
of smaller magnitudes, global solutions are found in around $80\%$ percentage of the simulations.
}
\end{example}

\begin{example}\label{exm:3}
{\em
The tensors in this example are randomly generated with each element in $[0,1]$. The best rank one approximation is computed, i.e., $r=1$. Examples with different dimension $n$ are simulated. For each $n\in\{2,\dots,10\}$, $100$ randomly generated instances are tested. We recorded the number of times that the algorithm successfully computed the best rank one approximation and certified this fact using the theory developed in this paper as well. We see from the table that the performance is quite promising.
\begin{table}[h!]\caption{Performance of randomly generated tensors.}\label{table:exm-3}
\centering\small
\begin{tabular}{|c | c | c | c| c   | c | }
\hline
$n$& \multicolumn{2}{c|}{2}& 3 &4 &5  \\\hline
\text{gap}& $5.490\times 10^{-8}$&0.391 &  $2.674\times 10^{-8}$&  $4.127\times 10^{-8}$&  $5.785\times 10^{-8}$\\ \hline
\text{Num}& 94&6& 100& 100& 100\\ \hline\hline
$n$&6&7&8&9&10\\ \hline
\text{gap}& $5.826\times 10^{-8}$&  $7.496\times 10^{-8}$&  $1.102\times 10^{-6}$&  $1.600\times 10^{-8}$&  $4.926\times 10^{-7}$\\ \hline
\text{Num}&  100 &100 &100&100&100\\ \hline
\end{tabular}
\end{table}
}
\end{example}

\section{Conclusions}\label{sec:conclusion}
In this paper, we presented a method for computing the best low rank approximation for a given third order symmetric tensor. It is shown that this method can certify the global optimality under mild assumptions by employing techniques from polynomial optimization, matrix optimization, duality theory, and nonsmooth analysis. The applicability of the theory is verified by several numerical examples.

The emphasis of this paper is on the global optimality and quantified quasi-optimality certification of the best low rank approximation. Numerical illustration is presented for the validation of the theory as well. However, more carefully and wisely designed numerical methods should be investigated in our future research for solving the hard optimization problems involved in the theory.
In particular, the method employed in this paper for the subproblem \eqref{eq:non-opt-rank-nu-linear} is a first order method, which
typically has a slow convergence and it is difficult to get a high accuracy solution. On the other hand, the polynomial optimization ingredients in our reformulation \eqref{eq:non-opt-rank-reform} require a high accuracy solution; otherwise, the flatness condition is impossible to be satisfied. Hence, methods for solving \eqref{eq:non-opt-rank-nu-linear} and \eqref{eq:non-opt-rank-nu-pen} with high accuracy and fast convergent properties should be developed.
In order to facilitate the global optimality certification, high accuracy solution for
the dual problem \eqref{eq:lagrange-dual-min} is also necessary.
But \eqref{eq:lagrange-dual-min} has a nonsmooth convex objective function, so
this is also a challenging problem to solve.

Nevertheless, the numerical examples in Section~\ref{sec:numerical} as well as the theoretical results in Section~\ref{sec:app-quality} on the global optimality certification convinced us that this approach is quite promising.

\backmatter


\bmhead{Acknowledgments}
Shenglong Hu is supported by the National Science Foundation of China under Grant 12171128 and the Natural Science Foundation of Zhejiang Province, China, under Grant LY22A010022. Defeng Sun is supported in part by RGC Senior Research Fellow Scheme (SRFS) under SRFS2223-5S02.  Kim-Chuan Toh is supported by the Ministry of Education, Singapore, under its Academic Research Fund Tier 3
grant call (MOE-2019-T3-1-010).

\begin{appendices}

\section{Indices and Moments}\label{sec:indices}
\subsection{Multisets of indices}\label{sec:matricization}
Let $\mathbb N$ be the set of natural numbers and $\mathbb N^n$ be the set of vectors of
dimension $n$ with components in $\mathbb N$.
Let $\mathbb I\in(\mathbb N^n)^n$ be defined as
\[
\mathbb I := (\e_1,\dots,\e_n).
\]
In the notation $\I$, the dependence on $n$ is omitted for simplicity as it should be clear
from the context.

We define $\mathbb I^1:=\mathbb I$ and $\mathbb I^0:=(\0)$, where $\0\in\mathbb N^n$ is the vector of all zeros.
For any positive integer $s>1$, $\mathbb I^s\in (\mathbb N^n)^{n^s}$ is defined as $\mathbb I*\mathbb I^{s-1}$, and
\[
\mathbb I*\mathbb I^{s-1}:=(\e_1*\mathbb I^{s-1},\dots,\e_n*\mathbb I^{s-1}),
\]
and if $\mathbb I^{s-1}$ is written as $(\uu_1,\dots,\uu_{n^{s-1}})$, then
\[
\e_i*\mathbb I^{s-1}:=(\e_i+\uu_1,\dots,\e_i+\uu_{n^{s-1}}).
\]

Note that when $s>1$, there are repeated elements in the ordered set $\mathbb I^s$.
Actually, it is easy to see that there are exactly ${n+s-1\choose n-1}$ distinct elements in $\I^s$.

For a nonnegative integer $s$, define
\[
\I^{\leq s}:=\I^0\oplus \I^1\oplus\dots\oplus \I^s.
\]
The cardinality of $\I^{\leq s}$ is denoted as
\begin{equation*}
\nu(n,s):=\frac{n^{s+1}-1}{n-1}.
\end{equation*}

We can view both $\I^s$ and $\I^{\leq s}$ as sets of vectors  in $\mathbb N^n$.  For notational simplicity, we will use the notation $\alpha\in \I^s$ or $\alpha\in \I^{\leq s}$
to refer to $\alpha$ as a vector in the set $\I^s$ or $\I^{\leq s}$ respectively.
\subsection{Moment matrices and extended moment matrices}\label{sec:moment-matrix}
Moment matrices are useful tools in the study of polynomial optimization, we refer to \cite{L-01,L-09,N-14,N-14-ATKMP,N-15,nie2023moment} and references therein for basic notions and advances of polynomial optimization.
Let
\begin{equation}\label{eq:monomial}
\x^s:=(1,x_1,\dots,x_n,x_1^2,\dots,x_n^s)^\mathsf{T}
\end{equation}
be the vector of monomials up to degree $s$ in the $n$ variables $x_1,\dots,x_n$ ordered lexicographically. The dimension of $\x^s$ is ${n+s\choose n}$.
Let
\[
\x^{[s]}=(x_1^s,x_1^{s-1}x_2,\dots,x_n^s)^\mathsf{T}
\]
be the sub-vector of $\x^s$ corresponding to the monomials of degree exactly $s$. The
dimension of $\x^{[s]}$ is
\begin{equation*}
\zeta(n,s):={n+s-1\choose n-1}.
\end{equation*}
Let
\begin{eqnarray*}
\mathbb N^n_{\leq s} &:=&\{\alpha\in\mathbb N^n\colon |\alpha|:=\alpha_1+\dots+\alpha_n\leq s\},
\\[3pt]
\mathbb N^n_{=s}&:=&\{\alpha\in\mathbb N^n\colon |\alpha|:=\alpha_1+\dots+\alpha_n=s\}.
\end{eqnarray*}

Note that for each given integer $s\geq 0$, there exists a set of mutually orthogonal symmetric matrices $A_\alpha\in\SS^2(\{0,1\}^{\zeta(n+1,s)})$ such that
\[
\x^s(\x^{s})^\mathsf{T}=\sum_{\alpha\in\mathbb N^n_{\leq 2s}}\x^\alpha A_\alpha.
\]

In the classical analysis of polynomial optimization, a moment matrix of order $s$ is a matrix $M\in\SS^2(\mathbb R^{\zeta(n+1,s)})$ in the form
\begin{equation}\label{eq:moment}
M:=\sum_{\alpha\in\mathbb N^n_{\leq 2s}}y_{\alpha} A_\alpha
\end{equation}
for a vector $\y\in\mathbb R^{\zeta(n+1,2s)}$ which is indexed by the vector of exponents of monomials in $\x^{2s}$. The matrix $M$ in \eqref{eq:moment} is denoted as $\M_s(\y)$,
and it is known as the \textit{$s$-th order moment matrix generated by $\y$}.

There is a natural one to one correspondence between a monomial $\x^\alpha$ and a vector in $\mathbb N^n$. The relation is indicated directly by the exponent vector $\alpha$ of the given monomial. Likewise, we can define
\begin{equation*}
\x^{\circ s}:=(\x^{\uu_1},\dots,\x^{\uu_{\nu(n,s)}})^\mathsf{T}
\end{equation*}
with
\[
\I^{\leq s}=:(\uu_1,\dots,\uu_{\nu(n,s)});
\]
and
\begin{equation}\label{eq:tensor-1}
\x^{\otimes s}:=(\x^{\vv_1},\dots,\x^{\vv_{n^s}})^\mathsf{T}
\end{equation}
with
\[
\I^s=:(\vv_1,\dots,\vv_{n^s}).
\]

Note that $\x^{\otimes s}$ is used to refer to the symmetric rank one tensor of order $s$ generated by $\x$ as well. We hope that this abuse of notation will not bring confusion, since \eqref{eq:tensor-1} is consistent with the classical meaning and it is actually a vectorization of the rank one tensor $\x^{\otimes s}$ in the lexicographic order. The exact meaning would be clear from the context.

It is thus easy to see that the vector  $\x^s$ is a sub-vector of $\x^{\circ s}$,
and $\x^{[s]}$ is a sub-vector of $\x^{\otimes s}$.

Likewise, for each given integer $s\geq 0$, there exists a set of mutually orthogonal symmetric matrices $B_\alpha\in \SS^2(\{0,1\}^{\nu(n,s)})$ such that
\[
\x^{\circ s}(\x^{\circ s})^\mathsf{T}=\sum_{\alpha\in\mathbb N^n_{\leq 2s}}\x^\alpha B_\alpha.
\]

An \textit{extended moment matrix}\footnote{The employment of extended moment matrices can be avoided if an appropriate weight is chosen. However, after comparing it with the current presentation, we find the current one is clearer.} of order $s$ is a matrix $G\in \SS^2(\mathbb R^{\nu(n,s)})$ of the form
\begin{equation}\label{eq:moment-extend}
G:=\sum_{\alpha\in\mathbb N^n_{\leq 2s}}y_{\alpha} B_\alpha
\end{equation}
for a vector $\y\in\mathbb R^{\zeta(n+1,2s)}$ which is indexed by the vector of exponents of monomials in $\x^{2s}$. The matrix $G$ in \eqref{eq:moment-extend} is denoted as $\G_s(\y)$, which is called the \textit{$s$-th order extended moment matrix generated by $\y$}.

It is easy to see that a moment matrix of order $s$ generated by $\y$ is a principal sub-matrix of the extended moment matrix of order $s$ generated by the same $\y$. Moreover, we have the following result.
\begin{lemma}\label{lem:moment}
Let nonnegative integer $s$ be given and $\y\in\mathbb R^{\zeta(n+1,2s)}$. Let $M_1\in\SS^2(\mathbb R^{\zeta(n+1,s)})$ and $M_2\in \SS^2(\mathbb R^{\nu(n,s)})$ be respectively the moment matrix of order $s$ and the extended moment matrix of order $s$ generated by $\y$. Then, $M_1$ is a principal sub-matrix of $M_2$ and there exists a nonsingular matrix $P\in \SS^2(\mathbb R^{\nu(n,s)})$ such that
\[
P^\mathsf{T}M_2P = \begin{bmatrix}M_1&0\\ 0&0\end{bmatrix}.
\]
Thus, $\rk(M_1)=\rk(M_2)$ and $M_1\succeq 0$ if and only if $M_2\succeq 0$.
\end{lemma}

\begin{proof}
The result follows straightforwardly from the definitions.
\end{proof}

In fact, there is a permutation matrix $L$ such that $LM_2L^\mathsf{T}$ is a \textit{flat extension} of $M_1$.

\subsection{Moment Tensors and Localizing Matrices}\label{app:moment}
Let $\mathbf y\in\mathbb R^{\mathbb N^m}$ be a moment sequence. The (infinite) \textit{moment matrix} $M(\mathbf y)$ is defined element-wisely as
\[
(M(\mathbf y))_{\alpha,\beta}:=y_{\alpha+\beta}.
\]
The \textit{moment matrix of order $k$} is then defined as the leading $|\mathbb N^m_{\leq k}|\times|\mathbb N^m_{\leq k}|$ principal sub-matrix of the moment matrix $M(\mathbf y)$.

Likewise, the \textit{moment tensor} $\mathcal T(\mathbf y)$ is defined element-wisely as
\[
\big(\mathcal T(\mathbf y)\big)_{\alpha,\beta,\gamma}:=y_{\alpha+\beta+\gamma}.
\]
The \textit{moment tensor of order $p, q, r$}, denoted as $\mathcal T_{p,q,r}(\mathbf y)$, is then defined as the leading $|\mathbb N^m_{\leq p}|\times|\mathbb N^m_{\leq q}|\times|\mathbb N^m_{\leq r}|$ principal sub-tensor of the moment tensor $\mathcal T(\mathbf y)$.

Given a polynomial $g(\mathbf x)\in\mathbb R[\mathbf x]_r$, the \textit{localizing matrix} $L^{k}_g(\y)$ of order $k$ is given by
\[
\mathbf p^\mathsf{T}L^{k}_g(\y)\mathbf p=\langle\mathbf y,p^2g\rangle\ \text{for all }p(\mathbf x)\in\mathbb R[\mathbf x]_k.
\]

\begin{proposition}[Localizing Matrix via Moment Tensor]\label{prop:moment-tensor-local}
For any given polynomial $g(\mathbf x)\in\mathbb R[\mathbf x]_r$, it always holds that
\[
L^{k}_g(\y)=\langle\mathcal T_{k,k,r}(\mathbf y),\mathbf g\rangle_{3:1}\in\mathbb R^{|\mathbb N^m_{\leq k}|\times|\mathbb N^m_{\leq k}|},
\]
where $\big(\langle\mathcal T_{k,k,r}(\mathbf y),\mathbf g\rangle_{3:1}\big)_{\alpha,\beta}:=\sum_{\gamma\in \mathbb N^m_{\leq r}}\big(\mathcal T_{k,k,r}(\mathbf y)\big)_{\alpha,\beta,\gamma}g_{\gamma}$ for all $\alpha,\beta$.
\end{proposition}

\begin{proof}
It follows that
\[
\langle\mathbf y,p^2g\rangle=\sum_{\alpha,\beta,\gamma}y_{\alpha+\beta+\gamma}p_{\alpha}p_{\beta}g_{\gamma}=\langle\mathcal T_{k,k,r}(\mathbf y),\mathbf{p}\otimes\mathbf{p}\otimes\mathbf{g}\rangle.
\]
The result thus follows.
\end{proof}

\subsection{Flatness}\label{sec:flatness}
In this subsection, we review basic facts about flatness of truncated moment sequence over the unit sphere $\mathbb S^{n-1}$ (abbreviated as \textit{utms}). For $k\geq 2$, a utms $\y\in\mathbb R^{\zeta(n+1,2k)}$ is \textit{flat} if (cf.\ \cite{N-14-ATKMP})
\begin{equation}\label{eq:flatness}
\mathcal M_k(\y)\succeq 0,\ \mathcal L_k(\y)=0,\ \text{and }\operatorname{rank}(\mathcal M_k(\mathbf y))=\operatorname{rank}(\mathcal M_{k-1}(\mathbf y)),
\end{equation}
where $\mathcal L_k(\y):=L^{k-1}_{1-\x^\mathsf{T}\x}(\y)$ is the $(k-1)$-th localizing matrix of the polynomial $1-\mathbf x^\mathsf{T}\mathbf x$. We refer to Appendix~\ref{app:moment} for the notion of localizing matrices.

To be more precise, the condition \eqref{eq:flatness} is called \textit{the $k$-th flatness condition} for the utms. If $\y$ satisfies the $k$-th flatness condition, then $\y$ can be represented as a unique measure which is $\operatorname{rank}(\mathcal M_k(\mathbf y))$-atomic \cite{N-14-ATKMP,Curto-Fialkow:tkm}. We will call the cardinality of the support of this unique measure the \textit{rank of the utms}, denoted as $\operatorname{rank}(\y)$. Thus, in this case, $\operatorname{rank}(\y)=\operatorname{rank}(\mathcal M_k(\mathbf y))$.
If $\y$ does not satisfy the $k$-th flatness condition but some extension $\z$ of $\y$ satisfies the $s$-th flatness condition with $s>k$, then $\y$ can also be represented as a unique measure which is $\operatorname{rank}(\mathcal M_s(\mathbf z))$-atomic, and $\rk(\y):=\rk(\z)$. Since $\mathcal M_k(\mathbf y)$ is a principal sub-matrix of $\mathcal M_s(\mathbf z)$, {it may happen that $\operatorname{rank}(\y)=\operatorname{rank}(\z)=\operatorname{rank}(\mathcal M_s(\mathbf z))>\operatorname{rank}(\mathcal M_k(\mathbf y))$. } 
\section{Nonsmooth Analysis of Matrix Low Rank Projection}\label{app:low-rank}

Let positive integers $m\leq n$.
Given a matrix $X\in\mathbb R^{m\times n}$ and a positive integer $r\leq m$, we consider the following problem on projection of $X$ onto the set $\operatorname{R}(r)$ of matrices of rank at most $r$ in the ambient space $\mathbb R^{m\times n}$, i.e.,
\begin{equation}\label{app-eq:low-rank}
\begin{array}{rl}\min&\frac{1}{2}\|Y-X\|^2\\ \text{s.t.}& \operatorname{rank}(Y)\leq r,\\
&Y\in \mathbb R^{m\times n}.\end{array}
\end{equation}
It is well-known that an optimizer of \eqref{app-eq:low-rank} can be computed via \textit{singular value decomposition} by Eckart-Young-Mirsky's theorem \cite{GV-12}.
Actually,
let
\[
X=P\Sigma(X)Q^\mathsf{T}
\]
be the singular value decomposition of $X$ with an orthogonal matrix $P\in\mathbb R^{m\times m}$ and an orthonormal $Q\in\mathbb R^{n\times m}$, and a diagonal matrix $\Sigma=\operatorname{diag}\{\sigma_1,\dots,\sigma_m\}$ with the singular values being ordered in nonincreasing order. In the sequel, we follow \cite{gao2010structured,GS-10} for the nonsmooth analysis of the matrix low rank projection.
We can partition the index set as
\[
\alpha:=\{i\colon \sigma_i>\sigma_r\},\ \beta:=\{i\colon \sigma_i=\sigma_r\}\ \text{and }\gamma:=\{i\colon \sigma_i<\sigma_r\}.
\]
Let $\Pi_{\operatorname{R}(r)}(X)$ be the set of optimizers of problem \eqref{app-eq:low-rank}. In generic case, the set $\Pi_{\operatorname{R}(r)}(X)$ is a singleton, while in some cases, it is a smooth manifold of dimension greater than one.
Nevertheless, each optimizer of \eqref{app-eq:low-rank} can be written as
\begin{equation*}
\begin{bmatrix}P_\alpha&P_\beta U_\beta\end{bmatrix}\operatorname{diag}(\vv)\begin{bmatrix}Q_\alpha&Q_\beta U_\beta\end{bmatrix}^\mathsf{T}
\end{equation*}
with an orthogonal matrix $U_\beta\in \mathbb O(|\beta|)$ and $\vv\in V$ with
\begin{equation}\label{eq:set-v}
\small
V:=\Bigg\{\vv\in\mathbb R^{|\alpha|+|\beta|}\colon v_i=\begin{cases}\sigma_i& \begin{array}{c}\text{for all }i\in\alpha\cup\beta^*\ \text{with }\beta^*\subseteq\beta\ \\ \text{and }|\beta^*|=r-|\alpha|,\end{array}\\
0& \text{for the others}\end{cases}\Bigg\}.
\end{equation}

It is a direct calculation to check that
\[
\|X-Y\|^2=\sum_{i=r+1}^m\sigma_i^2
\]
is a constant for all $Y\in \Pi_{\operatorname{R}(r)}(X)$, if it is not a singleton. Thus, we will (in some sense abuse of notation) use
\[
\|X-\Pi_{\operatorname{R}(r)}(X)\|^2
\]
to denote the above constant. Similar convention is taken in some other situation as well.

Let
\begin{equation*}
\Theta_r(X):=\frac{1}{2}\|\Pi_{\operatorname{R}(r)}(X)\|^2.
\end{equation*}
We then have
\begin{align*}
\Theta_r(X)&=\frac{1}{2}\|X\|^2-\frac{1}{2}\|X-\Pi_{\operatorname{R}(r)}(X)\|^2\\
&=\frac{1}{2}\|X\|^2-\min_{Y\in\operatorname{R}(r)}\frac{1}{2}\|X-Y\|^2\\
&=\max_{Y\in\operatorname{R}(r)}\Big\{\frac{1}{2}\|X\|^2-\frac{1}{2}\|X-Y\|^2\Big\}\\
&=\max_{Y\in\operatorname{R}(r)}\big\{\langle X,Y\rangle-\frac{1}{2}\|Y\|^2\big\},
\end{align*}
which shows that $\Theta_r$ is a convex function.

As a convex function, we can compute its subdifferentials \cite{R-70}. Given a subset $S$, $\operatorname{conv}(S)$ denotes its convex hull in the ambient space. The next result follows from \cite[Proposition~2.16]{gao2010structured}.
\begin{lemma}\label{lem:subdiff}
We have
\begin{equation}\label{eq:subdiff}
\partial\Theta_r(X)=\operatorname{conv}\big(\Pi_{\operatorname{R}(r)}(X)\big).
\end{equation}
\end{lemma}

\begin{lemma}\label{lem:low-rank}
Given a matrix $X\in\mathbb R^{m\times n}$ and positive integer $r\leq m$. If $Y\in\operatorname{conv}(\Pi_{\operatorname{R}(r)}(X))$, then
\begin{equation*}
Y\in \Pi_{\operatorname{R}(r)}(X)\ \text{if and only if }\operatorname{rank}(Y)\leq r.
\end{equation*}
\end{lemma}

\begin{proof}
Let
\[
X=P\Sigma(X)Q^\mathsf{T}
\]
be the singular value decomposition of $X$ with $\Sigma=\operatorname{diag}\{\sigma_1,\dots,\sigma_m\}$ consisting of nonincreasingly ordered singular values.
We can partition the index set as
\[
\alpha:=\{i\colon \sigma_i>\sigma_r\},\ \beta:=\{i\colon \sigma_i=\sigma_r\}\ \text{and }\gamma:=\{i\colon \sigma_i<\sigma_r\}.
\]
Each matrix $Z\in\Pi_{\operatorname{R}(r)}(X)$ takes the following form
\[
\begin{bmatrix}P_\alpha&P_\beta U_\beta\end{bmatrix}\operatorname{diag}(\vv)\begin{bmatrix}Q_\alpha&Q_\beta U_\beta\end{bmatrix}^\mathsf{T}
\]
with
\[
\vv\in\Bigg\{\vv\in\mathbb R^{|\alpha|+|\beta|}\colon v_i=\begin{cases}\sigma_i& \text{for all }i\in\alpha\cup\beta^*\ \text{with }\beta^*\subseteq\beta\ \text{and }|\beta^*|=r-|\alpha|,\\
0& \text{for the others}\end{cases}\Bigg\}.
\]

More concretely, it can be written as
\[
P^\mathsf{T}Z\tilde Q=\begin{bmatrix}\operatorname{diag}(\Sigma_\alpha)&0&0\\0&\sigma_rU_{\beta^*}U_{\beta^*}^\mathsf{T}&0\\0&0&0\end{bmatrix},
\]
where $U_{\beta^*}\in\mathbb R^{|\beta|\times |\beta^*|}$ is formed by the columns
of $U_{\beta}$
indexed by $\beta^*$, and $\tilde Q$ is an orthogonal matrix formed as $[Q\ \bar Q]$.
Let $Y\in \operatorname{conv}(\Pi_{\operatorname{R}(r)}(X))$. By Carath\'eodory's theorem \cite{R-70}, we can write
\[
Y=\sum_{s=1}^S\mu_sZ_s
\]
as a convex combination of $Z_s\in \Pi_{\operatorname{R}(r)}(X)$. We thus have
\begin{equation}\label{eq:app-decomp}
P^\mathsf{T}Y\tilde Q=\begin{bmatrix}\operatorname{diag}(\Sigma_\alpha)&0&0\\0&\sigma_r\sum_{s=1}^S\mu_sU^s_{(\beta^*)_s}(U^s_{(\beta^*)_s})^\mathsf{T}&0\\0&0&0\end{bmatrix},
\end{equation}
for orthogonal matrices $U^s$ and index sets $(\beta^*)_s$ with $s\in\{1,\dots,S\}$.
The case when $\sigma_r=0$ is trivial. In the following, we assume that $\sigma_r>0$.

Let
\[
p:=|(\beta^*)_s|\ \text{for all }s=1,\dots,S
\]
and $p=r-|\alpha|$.
Then
\[
\operatorname{rank}((U^s_{(\beta^*)_s})(U^s_{(\beta^*)_s})^\mathsf{T})=p, \quad
\forall \; s=1,\dots,S.
\]

By the assumption, we have
\[
\operatorname{rank}(Y)\leq r=|\alpha|+p.
\]
Therefore, we have from \eqref{eq:app-decomp} that
\[
\operatorname{rank}\Big(\sum_{s=1}^S\mu_s(U^s_{(\beta^*)_s})(U^s_{(\beta^*)_s})^\mathsf{T}\Big)=\operatorname{rank}((U^s_{(\beta^*)_s})(U^s_{(\beta^*)_s})^\mathsf{T})=p
\]
for each $s=1,\dots,S$. In fact, since each component matrix in the summation
\[
\sum_{s=1}^S\mu_s(U^s_{(\beta^*)_s})(U^s_{(\beta^*)_s})^\mathsf{T}
\]
is positive semidefinite and has all the eigenvalues being $0$ or $\mu_s$, we must have the component matrices are the same. Actually, we have
\[
\operatorname{ker}(\sum_{s=1}^S\mu_s(U^s_{(\beta^*)_s})(U^s_{(\beta^*)_s})^\mathsf{T})\subseteq\operatorname{ker}((U^1_{(\beta^*)_1})(U^1_{(\beta^*)_1})^\mathsf{T})
\]
whose dimensions equal to $|\beta|-p$, and thus the two kernels are equal to each other. Consequently, all the kernels
\[
\operatorname{ker}((U^s_{(\beta^*)_s})(U^s_{(\beta^*)_s})^\mathsf{T})\ \text{for all }s=1,\dots,S
\]
are the same. Let $W\in\mathbb R^{|\beta|\times(|\beta|-p)}$ be a matrix with orthonormal columns which form a basis for the common kernel. Then, we have for all $s=1,\dots,S$
\[
\begin{bmatrix}(Q^s_{(\beta^*)_s})&W\end{bmatrix}
\]
is an orthogonal matrix. Therefore, we have
\[
(U^s_{(\beta^*)_s})(U^s_{(\beta^*)_s})^\mathsf{T}=I-WW^\mathsf{T}\ \text{for all }s=1,\dots,S,
\]
which implies that all the matrices in the convex combination are the same.
Therefore,
\[
Z_1=\dots=Z_S.
\]
The conclusion then follows.
\end{proof}

By the proof, we see that the extreme points of $\operatorname{conv}(\Pi_{\operatorname{R}(r)}(X))$ are those in the set $\Pi_{\operatorname{R}(r)}(X)$, and can be characterized by the rank function.

A similar result for symmetric matrices can be proved similarly, we state it here for its independent interest.
Let $\operatorname{S}(r)$ be the set of symmetric matrices of rank at most $r$.
\begin{proposition}\label{lem:low-rank-sym}
Given a matrix $X\in\operatorname{S}^2(\mathbb R^n)$ and a positive integer $r\leq n$. If $Y\in\operatorname{conv}(\Pi_{\operatorname{S}(r)}(X))$, then
\begin{equation*}
Y\in \Pi_{\operatorname{S}(r)}(X)\ \text{if and only if }\operatorname{rank}(Y)\leq r.
\end{equation*}
\end{proposition}

\end{appendices}


\bibliographystyle{abbrv}
\bibliography{references}

\end{document}